\documentclass[12pt]{article}
\usepackage[a4paper,left=0.7in,top=1in,right=0.7in,bottom=1in]{geometry}

\usepackage[utf8]{inputenc}
\usepackage{algorithm}
\usepackage[noend]{algpseudocode}
\usepackage{enumitem}
\usepackage{amsmath,amssymb,amsfonts,amsthm,latexsym}
\usepackage{mathrsfs}
\usepackage{graphicx}
\usepackage{color}
\usepackage{mathtools}
\usepackage[normalem]{ulem}
\usepackage{scalerel}
\usepackage[mathcal]{euscript}
\usepackage{algorithm}
\usepackage[noend]{algpseudocode}
\usepackage{xspace}
\usepackage{ifthen}

\usepackage{url}

\usepackage{authblk}
\usepackage[numbers, sort&compress]{natbib}

\usepackage[colorlinks]{hyperref}
\hypersetup{
	citecolor=blue,
   linkcolor=black,
}
\usepackage[font=scriptsize,width=.85\textwidth,labelfont=bf]{caption}
\usepackage{mathptmx}

\newenvironment{owndesc}%
{\begin{description}[leftmargin = 0.2cm, labelsep = 0.2cm]}
  {\end{description}}

\usepackage{color}
\newcommand{\mc}{\mathcal}

\newcommand{\lc}[1]{\ensuremath{_{#1}}} %% scriptscriptsize TOO SMALL!!!!!!

\newcommand{\med}{\ensuremath{\mathrm{med}}}
\newcommand{\CQinfer}{\ensuremath{C^\textnormal{iso}\lc{6}}-\ensuremath{Q\lc{3}}-inferring\xspace}
\newcommand{\QQinfer}{\ensuremath{Q^-\lc{3}}-\ensuremath{Q\lc{3}}-inferring\xspace}
\newcommand{\cd}[2]{\ensuremath{\,\lc{#1}\!\circledast\lc{#2}\!}}

\newcommand{\Gin}[1]{G^\textnormal{in}\lc{#1}}
\newcommand{\Gout}[1]{G^\textnormal{out}\lc{#1}}
\newcommand{\Hin}[1]{H^\textnormal{in}\lc{#1}}
\newcommand{\Hout}[1]{H^\textnormal{out}\lc{#1}}
\newcommand{\GinS}{G^\textnormal{in}}
\newcommand{\GoutS}	{G^\textnormal{out}}

\usepackage{cleveref}

\newtheorem{theorem}{Theorem}[section]
\newtheorem{lemma}[theorem]{Lemma}
\newtheorem{proposition}[theorem]{Proposition}
\newtheorem{corollary}[theorem]{Corollary}
\theoremstyle{definition}
\newtheorem{definition}[theorem]{Definition}
\newtheorem{obs}[theorem]{Observation}

\newtheorem{remark}[theorem]{Remark}

\newtheorem*{conjecture}{Conjecture}

\crefname{theorem}{Theorem}{Thm.}
\Crefname{theorem}{Theorem}{Theorems}
\crefname{lemma}{Lemma}{Lemma}
\Crefname{lemma}{Lemma}{Lemmas}
\crefname{proposition}{Proposition}{Prop.}
\Crefname{proposition}{Proposition}{Propositions}
\crefname{corollary}{Corollary}{Cor.}
\Crefname{corollary}{Corollary}{Corollaries}
\crefname{definition}{Def.}{Def.}
\Crefname{definition}{Definition}{Definitions}
\crefname{figure}{Fig.}{Fig.}
\Crefname{figure}{Figure}{Figures}
\crefname{equation}{Equ.}{Equ.}
\Crefname{equation}{Equation}{Equations}
\crefname{example}{Example}{Exa.}
\Crefname{example}{Example}{Examples}
\crefname{algorithm}{Algorithm}{Alg.}
\Crefname{algorithm}{Algorithm}{Algorithms}
\crefname{obs}{Observation}{Obs.}
\Crefname{obs}{Observation}{Observations}
\crefname{line}{Line}{Line}
\Crefname{line}{Line}{Lines}
\crefname{section}{Section}{Sec.}
\Crefname{section}{Section}{Sections}
\renewcommand{\theenumi}{(\arabic{enumi})}
\renewcommand{\theenumii}{(\alph{enumii})}

\makeatletter
\renewcommand\p@enumii{\theenumi}
\renewcommand\p@enumiii{\theenumi\theenumii}
\makeatother

\providecommand{\keywords}[1]{\textbf{\textit{Keywords: }} #1}

\makeindex             % used for the subject indexs

\title{Planar Median Graphs and Cubesquare-Graphs}
\author[1-2]{Carsten R.\ Seemann}
\author[3]{Vincent Moulton}
\author[1,4-7]{Peter F.\ Stadler}
\author[8]{Marc Hellmuth}

\affil[1]{Max Planck Institute for Mathematics in the Sciences,
	Inselstra{\ss}e 22, D-04103 Leipzig, Germany}

\affil[2]{Bioinformatics Group, Department of Computer Science and
	Interdisciplinary Center for Bioinformatics, University of Leipzig,
	H{\"a}rtelstra{\ss}e 16-18, D-04107 Leipzig, Germany}

\affil[3]{School of Computing Sciences, University of East Anglia,
	Norwich, NR4 7TJ, UK}

\affil[4]{Bioinformatics Group, Department of Computer Science;
	Interdisciplinary Center for Bioinformatics; German Centre for
	Integrative Biodiversity Research (iDiv) Halle-Jena-Leipzig; Competence
	Center for Scalable Data Services and Solutions Dresden-Leipzig; Leipzig
	Research Center for Civilization Diseases; and Centre for Biotechnology
	and Biomedicine, Leipzig University, H{\"a}rtelstra{\ss}e 16-18,
	D-04107 Leipzig, Germany}

\affil[5]{Institute for Theoretical Chemistry, University of Vienna,
	W{\"a}hringerstra{\ss}e 17, A-1090 Wien, Austria}

\affil[6]{Facultad de Ciencias, Universidad Nacional de Colombia,
	Sede Bogot{\'a}, Colombia}

\affil[7]{The Santa Fe Institute, 1399 Hyde Park Rd., Santa Fe, NM
	87501, United States}

\affil[8]{Department of Mathematics, Faculty of Science, 
	Stockholm University, SE - 106 91 Stockholm, Sweden}

\begin{document}

\date{\ }

\maketitle

\abstract{ 
Median graphs are connected graphs in which for all three vertices there
is a unique vertex that belongs to shortest paths between each pair of
these three vertices.  In this paper we provide several novel
characterizations of \emph{planar} median graphs. More specifically, we
characterize when a planar graph $G$ is a median graph in terms of
forbidden subgraphs and the structure of isometric cycles in $G$, and
also in terms of subgraphs of $G$ that are contained inside and outside
of 4-cycles with respect to an arbitrary planar embedding of $G$. These
results lead us to a new characterization of planar median graphs in
terms of \emph{cubesquare-graphs} that is, graphs that can be obtained by
starting with cubes and square graphs, and iteratively replacing 4-cycle
boundaries (relative to some embedding) by cubes or square-graphs. As a
corollary we also show that a graph is planar median if and only if it
can be obtained from cubes and square-graphs by a sequence of
``square-boundary'' amalgamations. These considerations also lead to an
$\mc{O}(n\log n)$-time recognition algorithm to compute a decomposition
of a planar median graph with $n$ vertices into cubes and square-graphs.
}
\smallskip

\noindent

\sloppy

\section{Introduction}
A median graph is a connected graph, in which, for each triple of vertices
there exists a unique vertex, called the \emph{median}, simultaneously
lying on shortest paths between each pair of the triple
\cite{mulder1998metric}.  While the term \emph{median graph} was introduced
by \citet{Nebesky:71} in 1971, they have been studied at least since the
1940's \cite{Avann:61,Birkhoff:47}.  Today, a great deal is known about
median graphs including several characterizations, see
e.g. \cite{Bandelt:08,klavzar1999median}. Median graphs naturally arise in
several fields of mathematics, for example, in algebra
\cite{bandelt1983median}, metric graph theory \cite{Bandelt:08} and
geometry \cite{chepoi2000graphs}, and they have practical applications in
areas such as social choice theory \cite{Bandelt:84,Day:03}, phylogenetics
(where Buneman graphs are of relevance; see \cite{Dress:97}), and forensic
science \cite{parson2007empop}. It is therefore natural to develop
approaches to better understand structural properties of median graphs, as
well as their subclasses.

Special classes of median graphs include, for example, \emph{trees},
\emph{square-graphs} (see \Cref{sec:preli} and e.g.\ \cite{Bandelt2010})
and \emph{cube-free} median graphs
(e.g.\ \cite{chepoi2019distance}). Interestingly, to date the class of
\emph{planar} median graphs has received relatively little attention,
although it is natural to consider such graphs from both a mathematical and
an application oriented perspective (see e.g.\ \cite{BSH:21}). Indeed, in
contrast to the plethora of characterizations available for median graphs,
so far only one direct characterization of planar median graphs has been
established by Peterin in \cite{Peterin2006}. In addition, the only other
results concerning planar median graphs that we are aware of are an
Euler-type formula for planar, cube-free median graphs \cite[Corollary
5]{klavzar1998euler}, and an algorithm for deciding in $\mc{O}(|V|+|E|)$
time whether or not a graph $G=(V,E)$ is a planar median graph \cite[Cor.\
3.4]{Imrich1999}.

Before proceeding with stating our results, it is informative to briefly
recall Peterin's characterization for planar median graphs. Given a graph
$G$ and a connected subgraph $G'$ of $G$, the \emph{expansion} of $G$ with
respect to $G'$ is the graph $H$ obtained by attaching a disjoint copy
$G''$ of $G'$ to $G$ by adding edges between corresponding vertices of $G'$
and $G''$. Expansions play a key role in characterizing median graphs and
their relatives \cite{Mulder1990}. More specifically, defining an expansion
$H$ to be \emph{convex} if $G'$ is a convex subgraph of $G$, a graph is a
median graph if and only if it can be obtained from $K\lc{1}$ by a series
of convex expansions \cite{Mulder1978,Mulder1990}. Planar median graphs can
be characterized by further restricting expansions. Call $H$ a \emph{face
	expansion} if there is a planar embedding of $G$ such that all vertices
of $G'$ are incident with the same face of $G$. Then a graph is a planar
median graph if and only if can be obtained from an edge by a sequence of
convex face expansions \cite{Peterin2006}.

In this paper, we shall characterize planar median graphs in an alternative
way by considering amalgamations. This has the advantage of allowing us to
decompose the graph into simpler building blocks. A graph $G$ is said to be
an \emph{amalgam} of two induced subgraphs $G_1$ and $G_2$ if their union
is $G$ and their intersection $G_1\cap G_2$ is non-empty
\cite{BH:91}. Amalgamation procedures differ by requiring certain
properties of $G_1$ and $G_2$ as subgraphs of $G$ and constraints imposed
on their intersection $G_1\cap G_2$. For instance, every median graph, can
be obtained by successive \emph{convex} amalgamations starting with
hypercubes \cite{isbell1980median,vV:84}, i.e., $G_1$ and $G_2$ are convex
subgraphs of their amalgam $G$ along $G_1\cap G_2$. Note that amalgamations
and expansions are closely related for median graphs (see
e.g.\ \cite[Theorem 7]{mulder1998metric}). Similar amalgamation results have
been proven for quasi-median graphs \cite[Theorem 1]{bandelt1994quasi} and
pseudo-median graphs \cite[Theorem 18]{bandelt1994quasi} (in terms of
``gated'' amalgamations). In this paper, we shall show that planar median
graphs can be obtained by starting with cubes and square-graphs and
iteratively amalgamating along the boundary of certain faces in some planar
embedding of the resulting graphs. 
This gives new insights into the fundamental properties of planar median graphs as
the structure of the basic building blocks (square-graphs and cubes) is very
well-understood \cite{Bandelt2010,harary1988survey}.
As we shall discuss below, this
approach is related to the 2-face expansions that are used in
\cite{Desgranges:17,Peterin:08} to characterize planar partial cubes.

The rest of this paper is organized as follows. After introducing the necessary
notation and reviewing some relevant results from the literature in \cref{sec:preli},
in \cref{sec:planmed} we present two characterizations of median graphs amongst
planar graphs. The first one (Theorem~\ref{thm:planar-mg<->iso-cycles}) is given in
terms of forbidden subgraphs and isometric cycles of a planar graph; the second one
is given by the condition that every 4-cycle or square $C$ in an embedding of a
planar graph must divide the graph into a planar median graphs that lie inside and
outside of $C$ (Theorem~\ref{thm:if-inoutside->mg<->pmg}). The last results prompts
us to introduce an operation $\circledast$ that glues graphs together at boundary
squares. In particular, in \cref{sec:QS}, we introduce QS-graphs as those graphs that
can be constructed from cubes and square-graphs by iterative application of the
$\circledast$ operation. We then proceed to show that the QS-graphs are exactly the
planar median graphs that are not trees (Theorem~\ref{thm:mpg<->Qthree}). As a
corollary we then show that a graph is a planar median graph if and only if it can be
obtained from cubes and square-graphs by a sequence of square-boundary amalgamations
(Theorem~\ref{thm:amalgam}). Section~\cref{sec:algo} is devoted to deriving an
efficient algorithm for finding a sequence of $\circledast$ operations for
decomposing a planar median graph into its basic pieces (i.e. cubes and
square-graphs). In the last section we discuss some open problems and possible future
directions.

%%%%%%%%%%%%%%%%%%%%%%%%%%%%%%%%%

\section{Preliminaries} \label{sec:preli}

\paragraph{\textbf{Graphs}}
We consider undirected graphs $G=(V,E)$ with finite vertex set $V(G)=V$ and
edge set $E(G)=E\subseteq \binom{V}{2}$, i.e., without loops and multiple
edges. If $E = \binom{V}{2}$, the graph $G$ is \emph{complete} and denoted
by $K_{|V|}$.  A \emph{complete bipartite graph} $K\lc{m,n} = (V,E)$ is a
graph whose vertex set $V$ can be partitioned into two subsets $V\lc{1}$
and $V\lc{2}$ with $|V\lc{1}|=m$ and $|V\lc{2}|=n$ such that
$\{v,w\} \in E$ if and only if $v \in V_i$ and $w \in V_j$ with $i \ne j$.
We write $G'\subseteq G$ if $G'$ is a subgraph of $G$ and $G[W]$ for the
subgraph in $G$ that is induced by some subset $W\subseteq V$.  The
\emph{graph union $G\lc{1}\cup G\lc{2}$} (resp.\ \emph{graph intersection $G\lc{1}\cap G\lc{2}$}) of two graphs
$G\lc{1}=(V\lc{1},E\lc{1})$ and $G\lc{2}=(V\lc{2},E\lc{2})$ is the graph
$(V\lc{1}\cup V\lc{2},\, E\lc{1}\cup E\lc{2})$ (resp.\ $(V\lc{1}\cap V\lc{2},\, E\lc{1}\cap E\lc{2})$).  The graph $G-X$ with
$X\subseteq V$ is the graph obtained from $G$ after removal of the vertices
in $X$ and its incident edges.  A graph $G$ is
$\{G'\lc{1},\ldots,G'_m\}$-free if none of the graphs $G'_i$ is a subgraph
of $G$. For simplicity, we write that $G$ is $G'$-free instead of
$\{G'\}$-free.

A \emph{shortest} path between $v$ and $w$ in $G$ is denoted by
$P^\star_G(v,w)$.  The length $d\lc{G}(v,w)$ of a shortest path between two
vertices $v$ and $w$ is called \emph{distance of $v$ and $w$ (w.r.t.\
	$G$)}.  A subgraph $G'$ of $G$ is \emph{isometric} if
$d\lc{G'}(v,w) = d\lc{G}(v,w)$ for all vertices $v,w\in V(G')$, and
$G'\subseteq G$ is \emph{convex} if for any two vertices $v,w \in V(G')$
\emph{every} shortest path $P^\star\lc{G}(v,w)$ between $v$ and $w$ is a
subgraph of $G'$. Clearly, every convex subgraph of $G$ is an isometric and
induced subgraph of $G$.
A graph $G$ is  \emph{$k$-connected} (for $k \in \mathbb{N}$) if
$|V(G)| > k$ and $G-X$ is connected for every set $X \subseteq V$ with
$|X| < k$. 

A \emph{cycle} is a connected graph in which every vertex has degree two.
The length of a cycle $C$ is the number of edges or equivalently, the
number of vertices in $C$.  A cycle $C\lc{n}$ of length $n\ge 3$ is called
an \emph{$n$-cycle}.  A $4$-cycle is also called a \emph{square}.  A graph
that does not contain cycles is \emph{acyclic} and, otherwise,
\emph{cyclic}.  A \emph{cogwheel $M_n$} consists of a cycle $C_n$ where
$n\ge 8$ is even and a ``central'' vertex that is adjacent to every second
vertex of this cycle. A \emph{suspended cogwheel} $M_n^*$ is obtained from
the cogwheel $M_n$ by adding an additional vertex adjacent to the central
vertex of $M_n$.

A connected acyclic graph $T=(V,E)$ is a \emph{tree}. A tree is
\emph{rooted} if there is a distinguished vertex $\rho\in V$ called the \emph{root of
	$T$}.  A \emph{(rooted) forest} is a graph whose connected components are
(rooted) trees.  For a rooted forest $T$, we say that vertex $v$ of $T$, is
\emph{at level $i$} if the distance from the root of the connected
component in $T$ that contains $v$ to vertex $v$ is precisely $i$. Hence,
all roots of the connected components of $T$ are at level $0$.

The \emph{Cartesian product $G\lc{1}\Box G\lc{2}$} of two graphs
$G\lc{1}=(V\lc{1},E\lc{1})$ and $G\lc{2}=(V\lc{2},E\lc{2})$ is the graph
with vertex set $V(G\lc{1}\Box G\lc{2})=V\lc{1}\times V\lc{2}$, and where
$\{(u,u'),\, (v,v')\} \in E(G\lc{1}\Box G\lc{2})$ precisely if either $u=v$
and $\{u',v'\}\in E\lc{2}$ or $u'=v'$ and $\{u,v\}\in E\lc{1}$.  The
Cartesian product is associative and commutative \cite{HammackEtAl2011},
which allows us to write $\Box_{i=1}^n G_i$ for the Cartesian product of
the graphs $G_1,\dots,G_n$.  An $n$-dimensional \emph{hypercube} $Q_n$ is
the $n$-fold Cartesian product $\Box_{i=1}^n K_2$.  A $Q\lc{3}$ is called
\emph{cube}.  The subgraph $Q^-\lc{3}$ of a cube $Q\lc{3}$ is obtained from
this $Q\lc{3}$ by removing one vertex and its incident edges.  A graph $G$
is \emph{\CQinfer} if for each isometric $C\lc{6}$ in $G$ there is a cube
$Q\lc{3} \subseteq G$ such that $C\lc{6}\subseteq Q\lc{3}$.  Analogously, a
graph $G$ is \emph{\QQinfer} if for each $Q^-\lc{3}$ in $G$ there is a cube
$Q\lc{3} \subseteq G$ such that $Q^-\lc{3}\subseteq Q\lc{3}$.

We now provide here a simple result for later reference.
\begin{lemma} \label{lem:C4convex} 
	Let $G$ be a $K\lc{3}$-free graph.
	Then, every $4$-cycle in $G$ is convex if and only if $G$ is
	$K\lc{2,3}$-free.
\end{lemma}
\begin{proof}
	Let $G=(V,E)$ be a $K\lc{3}$-free graph. By contraposition, assume that
	there is a $4$-cycle $C\subseteq G$ which is not convex.  Thus,
	there are two vertices $v,w \in V(C)$ for which there is a shortest
	path $P^\star_G(v,w)$ that is not contained in $C$.  Since
	$E\subseteq \binom{V}{2}$ and $G$ is $K\lc{3}$-free, $d_G(v,w)=2$.
	Consequently, the graph union $C \cup P^\star_G(v,w)$ forms a
	subgraph of $G$ that is isomorphic to a $K\lc{2,3}$.
	
	Conversely, assume that $G$ is not $K\lc{2,3}$-free.  Then, there is a square
	$C\subseteq K\lc{2,3} \subseteq G$ which is not convex.
\end{proof}

\paragraph{\textbf{Convex Hull and Shortest-Path-Extension (SPE)}}
For a subgraph $G'$ of $G$, the \emph{convex hull $\mc{H}(G')$ of $G'$
	(w.r.t.\ $G$)} is the intersection of all convex subgraphs $G''$ of $G$
with $G'\subseteq G''$.  Note that $\mathcal{H}(G')$ is a convex subgraph
of $G$ and that $\mathcal{H}(G') = G'$ for every convex subgraph $G'$ of
$G$.  A tool that will be useful in upcoming proof are
shortest-path-extensions.
\begin{definition} \label{def:SPE-seq-graph} Let $G'$ be some subgraph of
	$G$.  A \emph{shortest-path-extension (SPE) of $G'$ (w.r.t.\ $G$)} is
	obtained by the following procedure:
	\begin{enumerate}[noitemsep,nolistsep]
		\item Set $G'\lc{1} \coloneqq G'$, and set $i = 1$,
		\item If $G'_i$ is a convex subgraph of $G$, then we stop.
		Otherwise, there is a shortest path $P^\star_G(v,w)$ with
		$v,w \in V(G'\lc{i})$, which is not a subgraph of $G'_i$.  In this
		case, we set $G'_{i+1} \coloneqq G'_i \cup P^\star_G(v,w)$, increment
		$i$ and repeat Step 2.\smallskip
	\end{enumerate}
	Since we have
	$G'=G'\lc{1} \subsetneq G'\lc{2} \subsetneq G'\lc{3}\subsetneq \ldots
	\subseteq G$, and since $G$ is finite and convex (w.r.t.\ $G$), a
	shortest-path-extension of $G'$ must terminate.  We call the final
	sequence $\mathcal{S}(G') = (G'\lc{1}, G'\lc{2}, G'\lc{3},\ldots,G'_m)$,
	with $m\geq 1$, a \emph{SPE-sequence of $G'$} and the last graph $G'_m$
	in $\mathcal{S}(G')$ the \emph{SPE-graph of $G'$}.
\end{definition}

As shown next, the convex hull can be constructed by means of
SPE-sequences so that, in particular, the SPE-graph is well-defined.
\begin{lemma} \label{lem:charact-convexHull} Let $G$ be a graph and let
	$G'$ be a subgraph of $G$.  Then, the convex hull of $G'$ w.r.t.\ $G$ is
	equal to the SPE-graph of $G'$ w.r.t.\ $G$, and thus, the SPE-graph is
	unique.
\end{lemma}
\begin{proof}
	Let $G'\subseteq G$, $\mc{H}(G')$ be the convex hull of $G'$ (w.r.t.\
	$G$) and $(G'\lc{1}, G'\lc{2}, G'\lc{3},\ldots,G'_m)$ be an SPE-sequence
	of $G'$ (w.r.t.\ $G$).  Furthermore, let $H'$ be some convex subgraph of
	$G$ such that $G' \subseteq H'$.  Note, such subgraph $H'$ exists, since
	$G$ is convex (w.r.t.\ $G$) and therefore, we may set $H'\coloneqq G$.
	We use induction on $i \in \{1,\ldots,m\}$ to show that every $G'_i$ is a
	subgraph of $H'$, and hence $G'_m\subseteq H'$.  For the base case,
	\cref{def:SPE-seq-graph}~(1) implies $G'\lc{1}=G'\subseteq H'$.
	
	Now, let us assume that $G'_k\subseteq H'$ for every
	$k\in \{1,\ldots,i\}, 1\le k<m$, and consider the graph $G'_{i+1}$.  By
	\cref{def:SPE-seq-graph}~(2), $G'_i\subseteq G'_{i+1}$.  Since
	$G'_i\subseteq H'$ it remains to show that all
	$v \in V(G'_{i+1})\setminus V(G'_i)$ and
	$e \in E(G'_{i+1})\setminus E(G'_i)$ are also contained in $H'$.  By
	definition, $G'_{i+1} = G'_i \cup P^\star_G(w,w')$ for some
	$w,w' \in V(G'\lc{i})$ where $P^\star_G(w,w')$ is a shortest path which
	is not a subgraph of $G'_i$.
	Let $v \in V(G'_{i+1})\setminus V(G'_i)$ and
	$e \in E(G'_{i+1})\setminus E(G'_i)$.  By construction, $v$ and $e$ must
	be contained in $P^\star_G(w,w')$.  Since $H'$ is convex (w.r.t.\ $G$)
	and $w,w' \in V(G'_i)\subseteq V(H')$, this shortest path
	$P^\star_G(w,w')$ must be a subgraph of $H'$.  Therefore, $v \in V(H')$
	and $e \in E(H')$. Consequently, $V(G'_{i+1})\subseteq V(H')$ and
	$E(G'_{i+1})\subseteq E(H')$, and thus, $G'_{i+1}\subseteq H'$.  By
	induction, we have $G'_m \subseteq H'$.
	
	Finally, since $H'$ was chosen arbitrarily and $\mathcal{H}(G')$ is a
	convex subgraph of $G$, we may set $H' \coloneqq \mathcal{H}(G')$ and
	conclude that $G'_m \subseteq \mathcal{H}(G')$.  By definition, the
	SPE-graph $G'_m$ is a convex subgraph of $G$ and $G' \subseteq G'_m$.
	Therefore, we have, by definition of the convex hull,
	$\mc{H}(G') \subseteq \mc{H}(G'_m)= G'_m$. In summary, $\mc{H}(G')=G'_m$.
	Since the convex hull by definition is unique, $G'_m$ is unique as well.
\end{proof}

\paragraph{\textbf{Planar Graphs, Faces and Boundaries}}
A \emph{planar} graph $G$ can be embedded in the plane such that its edges
intersect only at their endpoints (in particular, only in case they are
incident with the same endpoints). Such embeddings are called \emph{planar
	embeddings} of $G$.  A planar graph $G$ together with a planar embedding
$\pi$ of $G$ is called \emph{$\pi$-embedded}.  

Let $G$ be a $\pi$-embedded planar graph.
The connected regions in
$\mathbb{R}^2$ of the complement of  $G$ are
called \emph{faces}.  One of these faces is unbounded in
$\mathbb{R}^2$ and is called the \emph{outer} face, while all
other faces are bounded in $\mathbb{R}^2$. These are called
\emph{inner} faces.
The subgraph of $G$ that encloses a face $F$ is said to \emph{bound} $F$
and is called the \emph{boundary} of $F$. If $G'\subseteq G$ bounds
an inner (resp., outer) face it is called \emph{inner} (resp., \emph{outer})
\emph{boundary} of $G$.  Note, by definition, boundaries
are not part of a face. However, a face is said to be \emph{incident} with
the vertices and edges of its boundary.  Correspondingly, the vertices 
of $G$ that are incident with the outer face are called
\emph{outer} vertices, and every other vertex, i.e., every vertex that is
not incident to the outer face is called \emph{inner} vertices. The set
$\mathring{V}(G)$ denotes the set of inner vertices of $G$. Note that 
outer vertices can be incident to  inner faces.  If $G$  
has different faces with the same boundary, then $G$ must be a cycle
\cite[Lemma~4.2.5]{Diestel2017}.  Note, this is the only case where the
inner and outer boundary coincide. In all other cases, different faces of a
planar embedded graph, have different boundaries.
A planar graph $G$ is \emph{outer-planar} if $G$ can be $\pi$-embedded in such a way 
that all vertices of $G$ are outer vertices (i.e., $\mathring{V}(G) = \emptyset$) \cite{Chartrand:67}. 
In particular, cycles, trees and $K\lc{2}\Box P_n$ are outer-planar graphs.

Every planar graph has, in particular, an embedding on a 2-sphere $\mathbb{S}^2$.  This
observation immediately implies that every bounded region can be chosen as
the outer face, see e.g.\ \cite[Sec.\ 4.3]{Diestel2017} for more
information.    We summarize the latter in
\begin{obs} \label{obs:cycle-bound}
	Let $G$ be a $\pi$-embedded planar
	graph and $G'\subseteq G$ be a boundary of $G$.  Then, there is a planar
	embedding of $G$ such that $G'$ is an inner boundary as well as a planar
	embedding of $G$ such that $G'$ is the outer boundary of $G$.
\end{obs}

It is well-known that an $n$-dimensional hypercube $Q_n$ is planar if and
only if $n\leq 3$, see e.g.\ \cite{harary1988survey}.  Since every subgraph
of a planar graph is planar as well, we obtain
\begin{lemma}\label{lem:planar-hypercube-3}
	For every hypercube $Q_n\subseteq G$ in a planar graph $G$ it holds that
	$n\le 3$.
\end{lemma}

Two planar embeddings
$\pi_1,\pi_2 \colon G \to \mathbb{S}^2$ are \emph{equivalent} if there is a
homeomorphism $h\colon \mathbb{S}^2 \to \mathbb{S}^2$ such that
$h\circ \pi_1 = \pi_2$. We say that a graph $G$ is \emph{uniquely embeddable} on
$\mathbb{S}^2$ (up to equivalence) if any two planar embeddings of $G$ on
$\mathbb{S}^2$ are equivalent.
\begin{theorem}[{\cite{Whitney:33}}]\label{thm:3c-p}
	Every 3-connected planar graph is uniquely embeddable on $\mathbb{S}^2$.
\end{theorem}

In the following, every square of $G$ that bounds a face for \emph{some}
planar embedding of $G$ is called \emph{square-boundary}.  We will denote
planar embeddings of cubes by $\rho$, see the graph $G_d$ in
\cref{fig:exam-cubesquare} for such a $\rho$-embedded cube.
\Cref{thm:3c-p} and the fact that cubes are 3-connected implies
\begin{obs}\label{obs:Cube-Embedding}
	Let $G$ be $\rho$-embedded cube. Then, all faces must be bounded by
	squares.  In particular, $C\subseteq G$ is a square if and only if $C$ is
	an inner or outer boundary in $G$ w.r.t.\ $\rho$, and thus, if and only if
	$C$ is a square-boundary of $G$.
\end{obs}

\begin{figure}[t]
	\begin{center}
		\includegraphics[width=0.8\textwidth]{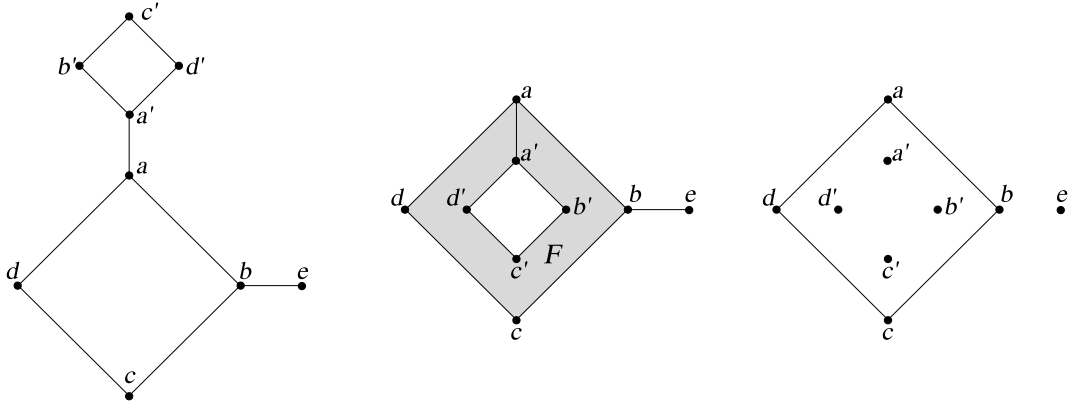}
	\end{center}
	\caption{Two distinct planar embeddings of a graph $G$ (left and
		middle).  \emph{Left:} $G$ is the outer boundary and every inner
		face is bounded by a square.  Since $G$ has no inner vertices, the
		conditions of Def.\ \ref{def:s-g} are satisfied and $G$ is a
		square-graph and, in particular, a planar median graph.  \emph{Middle:}
		$G' = G-e$ is an inner boundary which bounds the gray shaded face $F$.
		The graph induced by $\{a,b,c,d,e\}$ is the outer boundary and the
		4-cycle induced by $\{a',b',c',d'\}$ an inner boundary.  \emph{Right:}
		Consider the embedding $\pi$ of $G$ as in the middle figure. Then, shown
		is the $(G,\pi)$-induced embedding of the cycle $C$ induced by $\{a,b,c,d\}$
		and additionally, all other vertices of $G$.  W.r.t.\ the
		$(G,\pi)$-induced embedding of $C$, the vertices $a,b,c,d$ are almost-inside
		and almost-outside $C$, the vertex $e$ is outside (and, in particular,
		almost-outside) $C$ and the vertices $a',b',c',d'$ are inside (and in
		particular, almost-inside) of $C$.  Placing vertex $e$ inside $C$
		yields an embedding $\pi\lc{C}$ such that the outer boundary of $G$ is
		$C$ while placing the subgraph with vertices $a',b',c'$ and $d'$
		outside $C$ yields an embedding $\pi\lc{C}$ such that $C$ is an inner
		boundary in $G$.}
	\label{fig:planarDef}
\end{figure}

The following definitions are central for the presentation below.  Let $G$
be a planar graph and $G'\subseteq G$ be some subgraph of $G$.  Fixing a
planar embedding $\pi$ of $G$ and removing all edges and vertices from $G$
that are not contained in $G'$ yields a planar embedding of $G'$ that is
``anchored'' on the planar embedding $\pi$ of $G$.  We call such an
embedding of $G'$ a \emph{$(G,\pi)$-induced embedding}. Let $C$ be a cycle of $G$
and fix a planar embedding $\pi$ of $G$. A vertex $v\in V(G)$ is
\emph{outside} (resp., \emph{inside}) of $C$ if $v$ is contained in the
outer (resp., inner) face bounded by $C$ w.r.t.\ the $(G,\pi)$-induced
embedding.  By definition, $v \in V(C)$ is neither inside nor outside of
this $C$.  A vertex $v\in V(G)$ is \emph{almost-outside} (resp.,
\emph{almost-inside}) if either $v\in V(C)$ or $v$ is outside (resp.,
inside) of $C$.

Below, we will make frequent use of the subgraphs $\Gin{C,\pi}$ and
$\Gout{C,\pi}$ of $G$ defined as follows. Let $G$ be a $\pi$-embedded
planar graph and $C\subseteq G$ be a square in $G$.  Then, $\Gin{C,\pi}$
(resp., $\Gout{C,\pi}$) is the subgraph of $G$ that is obtained by deleting
every vertex and every edge of $G$ that is located in the outer (resp.\
inner) face bounded by $C$ w.r.t.\ the $(G,\pi)$-induced embedding of
$C$. In particular, for $K_3$-free graphs $G$, the subgraph $\Gin{C,\pi}$
(resp.\ $\Gout{C,\pi}$) is induced by all vertices that are almost-inside
(resp.\ almost-outside) of $C$.  Note that the vertices of $C$ are
contained in both $\Gin{C,\pi}$ and $\Gout{C,\pi}$.  Given the
$(G,\pi)$-induced planar embedding of $\Gin{C,\pi}$ and $\Gout{C,\pi}$, the
square $C$ is the outer boundary of $\Gin{C,\pi}$ and an inner boundary of
$\Gout{C,\pi}$.

The next result provides some insights on the location of vertices w.r.t.\
subgraphs $C_4$ of planar and $\{K\lc{3},K\lc{2,3}\}$-free graphs, which we
need for later reference.
\begin{lemma}\label{lem:in-out-C4}
	Let $G$ be a planar $\pi$-embedded and $\{K\lc{3},K\lc{2,3}\}$-free graph, and
	$C,C'\subseteq G$ be squares.  If $a,b \in V(C)$, then for every
	$c \in V(P^\star_G(a,b))$, we have $c \in V(C)$.  Moreover, $\Gin{C,\pi}$
	and $\Gout{C,\pi}$ as well as $\Gin{C,\pi}\cap \Gout{C',\pi}$ and
	$\Gout{C,\pi}\cap \Gout{C',\pi}$ and $\Gin{C,\pi}\cap \Gin{C',\pi}$ are
	convex subgraphs of $G$.
\end{lemma}
\begin{proof}
	Let $G$ be a planar $K\lc{2,3}$-free graph, $C\subseteq G$ be a
	square and $a,b\in V(C)$. Hence, 
	$d\lc{G}(a,b)\in \{1,2\}$. If $d\lc{G}(a,b)=1$, then there is nothing to show.
	Suppose that
	$d\lc{G}(a,b)=2$. In this case, there are two vertices
	$c\lc{1},c\lc{2}\in V(C)$ on two shortest path between $a$ and
	$b$. If there would be a third shortest path of length two, then $G$
	would contain a $K\lc{2,3}$, which is not possible. Hence, every vertex
	on the shortest paths between $a$ and $b$ must be part of $C$.
	
	Now, let $\pi$ be an arbitrary planar embedding of $G$ and consider the
	$(G,\pi)$-induced embedding of a square $C$ and let $a,b \in V(G)$ and
	$P^\star_G(a,b)$ be an arbitrary shortest path between $a$ and $b$.
	Assume that $a$ and $b$ are almost-inside of $C$.  Moreover, we assume
	for contradiction that there is some $c \in V(P^\star_G(a,b))$ that is
	outside of this $C$.  Since $G$ is planar, there must be two vertices
	$a',b' \in V(C)$, which are part of this $V(P^\star_G(a,b))$ such that
	$P^\star_G(a,b)=P^\star_G(a,a')\cup P^\star_G(a',b') \cup
	P^\star_G(b',b)$ and $c \in V(P^\star_G(a',b'))$, and with
	$P^\star_G(a,a')$, $P^\star_G(a',b')$ and $P^\star_G(b',b)$ being some
	shortest paths between $a,a'$ and $a',b'$ and $b,b'$, respectively.
	Since $a',b' \in V(C)$ and $P^\star_G(a',b')$ is a shortest path with
	$c \in V(P^\star_G(a',b'))$, we conclude by analogous arguments as above
	that $c \in V(C)$.  Hence, $c$ is not outside of $C$; a contradiction.
	Thus, every $c \in V(P^\star_G(a,b))$ is almost-inside of $C$.  Since $G$
	is $K\lc{3}$-free, every edge $\{v,w\}\in E(G)$ with
	$v,w \in V(\Gin{C,\pi})$ is also contained in $E(\Gin{C,\pi})$.  Taken
	the last two arguments together, $\Gin{C,\pi}$ is a convex subgraph of
	$G$.
	
	Analogous arguments show that $\Gout{C,\pi}$ is a convex subgraph of $G$.
	Since the intersection of convex subgraphs yields a convex subgraph (cf.\
	e.g.\ \cite[L.\ 5.2]{HLS14}), we conclude that
	$\Gin{C,\pi}\cap \Gout{C',\pi}$, $\Gout{C,\pi}\cap \Gout{C',\pi}$ and
	$\Gin{C,\pi}\cap \Gin{C',\pi}$ are convex subgraphs of $G$.
\end{proof}

The following result is a direct consequence of
\Cref{lem:in-out-C4,lem:charact-convexHull}.
\begin{lemma}\label{lem:in-out-C4-hull}
	Let $G$ be a planar $\pi$-embedded $\{K\lc{3},K\lc{2,3}\}$-free graph 
	which contains a square $C\subseteq G$.  Moreover, let $G'\subseteq G$ be a subgraph
	and $\mc{H}(G')$ be its convex hull (w.r.t.\ $G$).  If every
	$v \in V(G')$ is almost-inside (resp., almost-outside) of C, then every
	$v' \in V(\mc{H}(G'))$ is almost-inside (resp., almost-outside) of C,
	where inside and outside refer to the $(G,\pi)$-induced embedding of $C$.
\end{lemma}

Note that shortest paths on isometric cycles $C'\subseteq G$ connecting its
vertices must be shortest paths in the underlying graph $G$.  By
\Cref{lem:in-out-C4}, the graphs $\Gin{C,\pi}$ and $\Gout{C,\pi}$ are
convex subgraphs of planar $\pi$-embedded and $\{K\lc{3},K\lc{2,3}\}$-free
graphs.  Thus, every isometric cycles of such a graph must be entirely
contained in $\Gin{C,\pi}$ or $\Gout{C,\pi}$.  We summarize the latter
discussion in
\begin{lemma}\label{lem:iso->in/out-C4}
	Let $G$ be a planar $\pi$-embedded and $\{K_3, K_{2,3}\}$-free graph, 
	$C\subseteq G$ be a square and $C'$ be an isometric cycle of $G$.
	Then, all vertices in $V(C')\setminus V(C)$ are either inside or 
	outside
	of $C$ w.r.t.\ the
	$(G,\pi)$-induced embedding of $C$. 
\end{lemma}

\begin{lemma}\label{lem:Q3-planar-maxDist}
	Let $G$ be a planar graph that contains a cube $Q_3$ and let
	$u,v\in V(Q_3)$ such that $d_{Q\lc{3}}(v,w)=3$.  Then, $Q\lc{3}-\{v,w\}$
	results in a $6$-cycle $C$ and, for every planar
	embedding $\pi$ of $G$, $v$ is located in the inner face and $w$ in the
	outer face of $C$ or \emph{vice versa} w.r.t.\ $(G,\pi)$-induced
	embedding of $C$.
\end{lemma}
\begin{proof}
	Let $G$ be a planar graph that contains a cube $Q_3$ and let
	$v,w\in V(Q_3)$ such that $d_{Q\lc{3}}(v,w)=3$.  One easily observes that
	$Q\lc{3}-\{v,w\}$ results in a $6$-cycle $C$.  If both $v$ and $w$ are
	inside (resp., outside) w.r.t.\ $(G,\pi)$-induced embedding of $C$, then
	$C$ would be an outer (resp., inner) boundary of the cube $Q_3$.
	However, every boundary of a cube has to be a square (cf.\
	\cref{obs:Cube-Embedding}), and therefore, $v$ must be located in the
	inner face and $w$ in the outer face of $C$ or \emph{vice versa}.
\end{proof}

\begin{definition}[$k$-FS]
	A graph $G$ satisfies the \emph{$k$-face-square-property (w.r.t.\ $\pi$)}
	(\emph{$k$-FS}, for short) if there is a planar embedding $\pi$ of $G$
	such that at least $k$ faces are bounded by a squares.
\end{definition}
For instance, every square satisfies $2$-FS and every cube
satisfies $5$-FS.
%%%%%%%%%%%%%%%%%%%%%%%%%%%%%%%%%%%%%%%%%%%%%%%%%%%%

\paragraph{\textbf{Median Graphs and Square-Graphs}}

A vertex $x \in V(G)$ is a \emph{median} of three vertices $u,v,w\in V(G)$
if $d\lc{G}(u,x)+d\lc{G}(x,v)=d\lc{G}(u,v)$,
$d\lc{G}(v,x)+d\lc{G}(x,w)=d\lc{G}(v,w)$ and
$d\lc{G}(u,x)+d\lc{G}(x,w)=d\lc{G}(u,w)$.  A connected graph $G$ is a
\emph{median graph} if every triple of its vertices has a unique median.
In other words, $G$ is a median graph if, for all $u,v,w\in V(G)$, there is
a unique vertex that belongs to shortest paths between each pair of $u, v$
and $w$.  We denote the unique median of three vertices $u$, $v$ and $w$ in
a median graph $G$ by $\med\lc{G}(u,v,w)$.

For later reference, we summarize here some well-known properties of
median graphs, see \cite{Bandelt1982,klavzar1999median} and \cite[p.\ 198]{Mulder1978}.

\begin{proposition} \label{prop:convex-hull-charact} A
	connected graph $G$ is a median graph if and only if the convex hull of
	any isometric cycle of $G$ is a hypercube.
\end{proposition}

\begin{proposition}\label{lem:basic-median-props}
	For every median graph $G=(V,E)$ the following statements are satisfied.
	\begin{enumerate}[noitemsep]
		\item $G$ is bipartite;
		\item $G$ is $K\lc{2,3}$-free;
		\item (a) $G$ is an induced subgraph of an hypercube and thus, \\
		(b) every edge $e \in E$ that lies on some cycle must be contained in
		some $C\lc{4}$;
		\item (a) for every subgraph $G'$ of $G$, the convex hull
		$\mc{H}(G')$ (w.r.t.\ $G$) is a median graph and thus, \\
		(b) every convex subgraph of $G$ is a median graph.
	\end{enumerate}
\end{proposition}

\begin{figure}
	\begin{center}
		\includegraphics[width=0.7\textwidth]{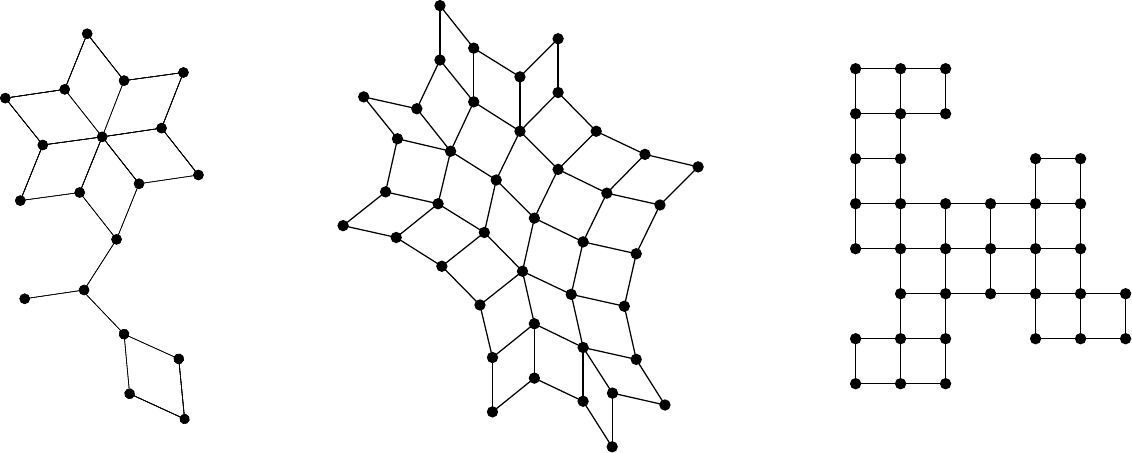}
	\end{center}
	\caption{Three example of square-graphs adapted from \cite[Fig.\ 1.1]{Bandelt2010}.
		\emph{Left:} A square-graph with three articulation
		points. \emph{Middle:} a $2$-connected square-graph. \emph{Right:} a
		so-called ``polyomino''.}
	\label{fig:exampl-square}
\end{figure}

The following type of graphs will play a crucial role for our results.
\begin{definition}\label{def:s-g}
	A \emph{square-graph} is a connected graph for which a planar embedding
	exists such that
	\begin{enumerate}[noitemsep,label=\textbf{(\alph*)}]
		\item every inner boundary is a square, and
		\item every inner vertex has at least degree $4$.
	\end{enumerate}
	Such a planar embedding of a square-graph will always be denoted by
	$\sigma$.
\end{definition}

Simple examples of square-graphs are trees and the 4-cycle.
Further examples of $\sigma$-embedded square-graphs are shown in
\cref{fig:exampl-square}. Below, we will make use of the
following results.
\begin{lemma}\label{lem:square-graph:basic-proper}
	Every square-graph as well as the cube $Q_3$ is a planar median graph
	(cf.\ \cite{Chepoi2002,Chepoi2004,SZP:73}).  Moreover, it can be decided
	in $\mc{O}(|V(G)|+|E(G)|)$ time whether a given graph $G$ is a
	square-graph or not (cf.\ \cite[Prop.\ 5.3]{Bandelt2010}).
\end{lemma}

It has been shown by \citet{SZP:73} and \citet[Prop.\ 9.1]{Bandelt2010}
that for every $\sigma$-embedded square-graph $G$ every square in $G$ is an
inner boundary.  This, together with the definition of square-graphs,
implies
\begin{lemma}\label{lem:square-faces-squaregraph}
	Let $G$ be $\sigma$-embedded square-graph. Then, $C\subseteq G$ is a
	square if and only if $C$ is an inner boundary in $G$ w.r.t.\ $\sigma$.
	Consequently, $C$ is a square if and only if $C$ is a square-boundary
	w.r.t.\ $\sigma$.  Moreover, every square-graph that contains $k$
	squares, satisfies $k$-FS (w.r.t.\ $\sigma$) and every cyclic square
	graph satisfies $1$-FS (w.r.t.\ $\sigma$).
\end{lemma}		

Recall that every face of planar graph can be both an inner and the outer
face depending on the choice of the embedding. This, together with
\cref{obs:cycle-bound} and the fact that in a $\sigma$-embedded
square-graph and $\rho$-embedded cube every square is square-boundary,
implies
\begin{obs}\label{obs:squaregraph-cube-embedding} 
	Let $G$ be a square-graph with planar embedding $\pi=\sigma$ or a cube
	with planar embedding $\pi=\rho$.  For every square $C$ of $G$ we can
	adjust $\pi$ to a planar embedding $\pi_C$ such that $C$ becomes an outer
	boundary while all other squares distinct from $C$ are inner boundaries
	w.r.t.\ $\pi_C$.  In case $G$ is a square-graph that contains at least
	two squares, there exists an inner face that is bounded by a square
	w.r.t.\ $\pi_C$.
\end{obs}

In particular, \citet{Bandelt2010} characterized square-graphs in terms of
forbidden subgraphs of median graphs.
\begin{proposition}[{\cite[Prop.\ 5.1 (i,
		ii)]{Bandelt2010}}] \label{prop:square-from-median} Let $G$ be a
	graph. Then, $G$ is a square-graph if and only if $G$ is a median graph
	such that $G$ does not contain any of the following graphs as induced
	subgraphs (or isometric subgraphs or convex subgraphs, respectively): the
	cube $Q\lc{3}$, the book $K\lc{2}\Box K\lc{1,3}$, and suspended
	cogwheel.
\end{proposition}

\begin{figure}
	\begin{center}
		\includegraphics[width=0.75\textwidth]{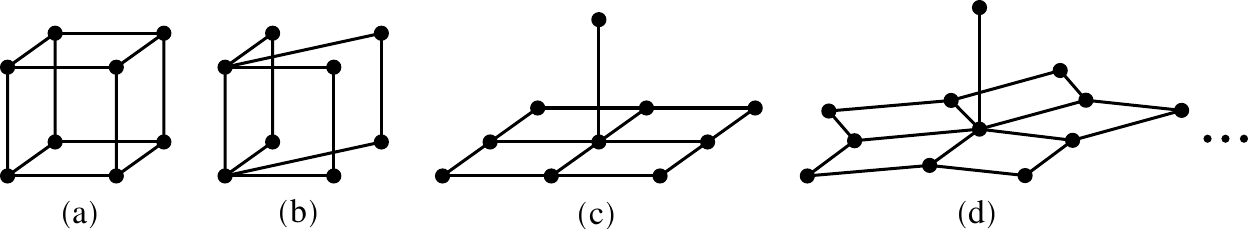}
	\end{center}
	\caption{Forbidden induced subgraphs of square-graphs adapted from \cite[Fig.\
		5.1]{Bandelt2010}. (a) Cube $Q\lc{3}$; (b) the book
		$K\lc{2}\Box K\lc{1,3}$; (c) and (d) the first two suspended cogwheels
		$M_8^*$ and $M_{10}^*$.}
	\label{fig:forb-subgraphs-square}
\end{figure}

%%%%%%%%%%%%%%%%%%%%%%%%%%%%%%%%%
\section{Characterization of Planar Graphs that are Median Graphs}
\label{sec:planmed}

In this section, we present new characterizations for planar graphs being median graphs.
To this end, we need the following
\begin{lemma} \label{lem:median->iso4-6} Let $G$ be a planar median
	graph. Then, the length of every isometric cycle of $G$ is either $4$ or
	$6$.
	\label{lem:length-isoC}
\end{lemma}
\begin{proof}
	Let $G$ be a planar median graph, and let $C_n$ with $n \ge 3$ be an
	isometric cycle of $G$.  First, we show $n \le 7$.  To this end, we
	assume for contradiction that $n \ge 8$.  Then,
	\cref{prop:convex-hull-charact} implies that there is a hypercube $Q_m$
	with $C_n \subseteq Q_m \subseteq G$, and \Cref{lem:planar-hypercube-3}
	implies that $m \le 3$.  Since $n \geq 8$, we have
	$C_n \subseteq Q\lc{3}$ and $n=8$. Thus, $V(C_8) = V(Q_3)$.  Since $C_n$
	contains two vertices at distance $\frac{n}{2}\ge 4$ and the diameter
	(i.e., the greatest distance) of $Q_3$ is $3$, this $C_n$ cannot be
	isometric; a contradiction.  Hence, $n \le 7$. Since every median graph
	is bipartite it cannot contain odd cycles.  Therefore, $n = 4$ or
	$n = 6$.
\end{proof}

\begin{lemma} \label{lem:trianglefree+K23-free->Q3-stuff} Let $G$ be a
	$\{K\lc{3},K\lc{2,3}\}$-free graph.  Moreover, let $Q\lc{3}\subseteq G$
	be a cube, and let $v,w \in V(Q\lc{3})$ with
	$d_{Q\lc{3}}(v,w)\in \{1,2\}$.  Then, for every shortest path
	$P^\star_G(v,w)$ in $G$, we have $P^\star_G(v,w)\subseteq Q\lc{3}$.
	Moreover, if $G$ is additionally planar, then
	$P^\star_G(v,w)\subseteq Q\lc{3}$ for all $v,w \in V(Q\lc{3})$.
\end{lemma}
\begin{proof}
	Let $G$ be a $\{K\lc{3},K\lc{2,3}\}$-free graph and $v,w \in V(Q\lc{3})$
	with $Q\lc{3}\subseteq G$ and $d_{Q\lc{3}}(v,w)\in\{1,2\}$.  Clearly, if
	$d_{Q\lc{3}}(v,w)=1$, then $P^\star_G(v,w)$ is an edge that must be
	contained in $Q\lc{3}$.  Now, assume that $d_{Q\lc{3}}(v,w)=2$. Then, it
	is easy to see that there is a (unique) square $C \subseteq Q\lc{3}$ with
	$v,w \in V(C)$.  Since $G$ is $\{K\lc{3},K\lc{2,3}\}$-free, we can apply
	\Cref{lem:C4convex} to conclude that $C$ is a convex subgraph of $G$.
	Hence, every shortest path $P^\star_G(v,w)$ between $v$ and $w$ is
	contained in a square $C\subseteq Q\lc{3}$.
	
	Now, assume that $G$ is planar in addition and let $\pi$ be an arbitrary
	planar embedding of $G$.  By the latter arguments, it suffices to
	consider vertices $v,w\in V(Q_3)$ with $d_{Q\lc{3}}(v,w)=3$.  By Lemma
	\ref{lem:Q3-planar-maxDist}, $Q\lc{3}-\{v,w\}$ results in a $6$-cycle $C$
	and $v$ is located in the inner face and $w$ in the outer face of $C$ or
	\emph{vice versa} w.r.t.\ the $(G,\pi)$-induced embedding of $C$.  This
	and the fact that $\pi$ is a planar embedding of $G$ implies that
	$P^\star_G(v,w)$ contains (at least) one vertex $u$ of
	$C$. Hence, there are shortest paths $P^\star_G(v,u)$ and
	$P^\star_G(u,w)$ such that
	$P^\star_G(v,w)=P^\star_G(v,u)\cup P^\star_G(u,w)$. 
	We distinguish 
	two (mutually exclusive) cases: \textit{(i)}
	$\{v,u\} \in E(Q\lc{3})\subseteq E(G)$ and \textit{(ii)}
	$\{u,w\}\in E(Q\lc{3})\subseteq E(G)$.
	
	In Case \textit{(i)} we have $P^\star_G(v,u) \subseteq Q\lc{3}$ and
	$d_{Q\lc{3}}(u,w)=2$.  By the latter arguments,
	$P^\star_G(u,w) \subseteq Q\lc{3}$.  Hence,
	$P^\star_G(v,w)=P^\star_G(v,u)\cup P^\star_G(u,w)\subseteq Q\lc{3}$.
	Similar arguments imply in Case \textit{(ii)} that
	$P^\star_G(v,w)\subseteq Q\lc{3}$.
\end{proof}

For later reference, we show that every isometric cycle $C\lc{6}$ and every
$Q^-\lc{3}$ of a median graph must be contained in a cube.
\begin{lemma}\label{lem:mg->Q3-Q3-infer}
	Every median graph is \QQinfer and \CQinfer.
\end{lemma}
\begin{proof}
	Let $G$ be a median graph. First, let $Q^-_3\subseteq G$. Consider the
	unique subgraph $C\lc{6} \subseteq Q^-_3$ that is an isometric subgraph
	in $Q^-_3$.  Assume, for contradiction, that this $C\lc{6}$ is not an
	isometric subgraph of $G$, i.e., that there are two vertices
	$a,b \in V(C\lc{6})$ with $d\lc{C\lc{6}}(a,b)>d\lc{G}(a,b)\ge 1$ and,
	therefore, $d\lc{C\lc{6}}(a,b)\in \{2,3\}$. If $d\lc{C\lc{6}}(a,b)=2$,
	then $d\lc{G}(a,b) = 1$. Thus, there must be a $K\lc{3}$ in $G$; a
	contradiction to \cref{lem:basic-median-props}~(1). Moreover, if
	$d\lc{C\lc{6}}(a,b)=3$ and $d\lc{G}(a,b)=2$, then $G$ must contain a
	$K_3$ or $C_5$; again a contradiction to
	\cref{lem:basic-median-props}~(1). Finally, assume that
	$d\lc{C\lc{6}}(a,b)=3$ and $d\lc{G}(a,b)=1$. Let $x,y,z \in V(C\lc{6})$
	be the three vertices that have degree $3$ in $Q^-_3$. Among these
	vertices $x,y,z$ has to be $a$ or $b$; w.l.o.g.\ assume that $x=a$. Note,
	there is a vertex $v \in V(Q^-_3)\setminus V(C_6)$ that is adjacent to
	every vertex in $\{x,y,z\}$.  Moreover, $b$ is adjacent to every vertex
	in $\{x,y,z\}$ in $G$, since $d\lc{G}(x,b)=1$. Hence, there is a
	$K\lc{2,3}$ with $V(K_{2,3}) = \{x,y,z,v,b\}$ in $G$; which is a
	contradiction to \cref{lem:basic-median-props}~(2). Thus, $C\lc{6}$ is an
	isometric subgraph of $G$.
	
	Thus, by \cref{prop:convex-hull-charact} there is a hypercube that
	contains $C_6$. Since we have $|V(C\lc{6})|=6$, we conclude that there is
	a cube $Q \subseteq G$ with $C\lc{6}\subseteq Q$.  Let $w$ be the unique
	vertex in $Q$ that is \emph{not} adjacent to $x,y$ and $z$ in $Q$ and let
	$x',y',z'$ be the three vertices in $C_6$ that are adjacent to $w$ in
	$Q$, where $x'$ is adjacent to $x$ and $y$, $y'$ is adjacent to $y$, and
	$z$ and $z'$ is adjacent to $x$ and $z$.  The graph $H$ with vertices
	$V(C)\cup \{v,w\}$ and edge set
	$\big\{\{x',x\}, \{x',y\}, \{y',y\}, \{y',z\}, \{z',x\}, \{z',z\},
	\{v,x\}, \{v,y\}, \{v,z\}, \{w,x'\}, \{w,y'\}, \{w,z'\}\big\}$ is by the
	preceding arguments a subgraph of $G$ and, in particular, a cube for
	which $H-\{w\}$ is equal to the graph $Q^-_3$ chosen at the beginning of
	this proof.  Hence, $G$ is \QQinfer.
	
	Now, let $C\lc{6}$ be an isometric cycle of $G$.  Then,
	\cref{prop:convex-hull-charact}, together with the previous arguments,
	imply that there is a cube $Q\lc{3}\subseteq G$ with
	$C\lc{6}\subseteq Q\lc{3} $. Hence, $G$ is \CQinfer.
\end{proof}

\begin{lemma} \label{lem:C6+Q3->conv-hull} Let $G$ be
	$\{K\lc{3},K\lc{2,3}\}$-free, planar and \CQinfer graph.  Then, the
	convex hull of every isometric cycle $C\lc{6}$ in $G$ is a $Q\lc{3}$.
	Moreover, if $C\lc{6}$ is an isometric cycle in a cube
	$Q\lc{3}\subseteq G$, then $C\lc{6}$ is an isometric cycle in
	$G$. 
\end{lemma}
\begin{proof}
	Let $G$ be chosen as in the statement and let $C\lc{6}$ be an arbitrary
	isometric cycle of $G$.  Since $G$ is \CQinfer there is a cube
	$Q\lc{3} \subseteq G$ such that $C\lc{6}\subseteq Q\lc{3}$.  Moreover,
	let $(G'\lc{1}, G'\lc{2}, G'\lc{3},\ldots,G'_m)$ be a SPE-sequence of
	$C\lc{6}$.  It is easy to verify that $Q\lc{3}\subseteq G'_m$.  Since $G$
	is $\{K\lc{3},K\lc{2,3}\}$-free and planar, we can apply
	\Cref{lem:trianglefree+K23-free->Q3-stuff} to conclude that
	$P^\star_G(v,w)\subseteq Q\lc{3}$ for every shortest path
	$P^\star_G(v,w)$ with $v,w \in V(Q\lc{3})$.  Hence,
	$Q\lc{3}\subseteq G'_m$ is a convex subgraph of $G$, and by construction
	of $(G'_1,\ldots,G'_m)$, we have $Q\lc{3}=G'_m$.  By
	\Cref{lem:charact-convexHull}, it follows that $Q\lc{3}$ is the convex
	hull of $C\lc{6}$ (w.r.t.\ $G$).
	
	Finally, let $C\lc{6}$ be an isometric cycle in a cube
	$Q\lc{3}\subseteq G$. Now, let $P^\star_G(v,w)$ be a shortest path with
	$v,w \in V(C\lc{6})\subseteq V(Q\lc{3})$. By the same arguments as above,
	$P^\star_G(v,w)\subseteq Q\lc{3}$ and $Q_3$ is a convex subgraph of
	$G$. Thus, every vertex (resp., edge) of $Q\lc{3}$ lies on some shortest
	path $P^\star_G(v,w)$ for all $v,w \in V(C\lc{6})$, we conclude that the
	convex hull $\mc{H}_G(C\lc{6})$ (w.r.t.\ $G$) is this $Q\lc{3}$. This,
	together with $C\lc{6}$ being an isometric cycle of that
	$Q\lc{3}\subseteq G$, implies $C\lc{6}$ is an isometric cycle in $G$.
\end{proof}

\begin{theorem} \label{thm:planar-mg<->iso-cycles} Let $G$ be a planar
	graph.  Then, $G$ is a median graph if and only if the following
	statements are satisfied:
	\begin{enumerate}[noitemsep,nolistsep]
		\item $G$ is connected, \label{item:conect1}
		\item $G$ is $K\lc{2,3}$-free, \label{item:K23-free}
		\item \CQinfer, and \label{item:C6-Q3-proper}
		\item every isometric cycle in $G$ has length $4$ or
		$6$. \label{item:iso-46}
	\end{enumerate}
\end{theorem}
\begin{proof}
	Let $G$ be a planar graph.  First, assume that $G$ is a median graph.
	Then, by definition \Cref{item:conect1} is satisfied,
	\Cref{lem:basic-median-props}~(2) implies \Cref{item:K23-free},
	\Cref{lem:mg->Q3-Q3-infer} implies \Cref{item:C6-Q3-proper}, and
	\Cref{lem:median->iso4-6} implies \Cref{item:iso-46}.
	
	Conversely, assume that
	\Cref{item:conect1,item:K23-free,item:C6-Q3-proper,item:iso-46} are
	satisfied.  Now, let $C\subseteq G$ be an arbitrary isometric cycle.  By
	\Cref{item:iso-46}, this cycle $C$ is either a $C\lc{4}$ or $C\lc{6}$.
	Note that $G$ is $K\lc{3}$-free, since any $K\lc{3}$ would be an
	isometric cycle.  If $C=C\lc{4}$, then we can apply \Cref{lem:C4convex}
	to conclude that $C\lc{4}$ is convex in $G$, and thus, the convex hull
	$\mc{H}(C\lc{4})$ is precisely this $C\lc{4}\simeq Q_2$.  If $C=C\lc{6}$,
	then \Cref{lem:C6+Q3->conv-hull} implies that the convex hull of this
	$C\lc{6}$ (w.r.t.\ $G$) is a cube $Q\lc{3}$.  Hence, in either case, the
	convex hull of any isometric cycle of $G$ is a hypercube.  Thus,
	\cref{prop:convex-hull-charact} implies that $G$ is a median graph.
\end{proof}

\begin{corollary}\label{cor:Q3free-med}
	Let $G$ be a planar graph.  Then, $G$ is a cube-free median graph if and
	only if the following statements are satisfied:
	\begin{enumerate}[noitemsep,nolistsep]
		\item $G$ is connected, \label{item:conect2}
		\item $G$ is $K\lc{2,3}$-free, \label{item:K23-free2}
		\item every isometric cycle in $G$ has length $4$. \label{item:iso-4}
	\end{enumerate}
\end{corollary}
\begin{proof}
	Let $G$ be a planar and cube-free median graph.  Then, by definition,
	\Cref{item:conect2} is satisfied, and \Cref{lem:basic-median-props}~(2)
	implies \Cref{item:K23-free2}.  Moreover, every isometric cycle in $G$
	has length $4$ or $6$ (cf.\ \cref{lem:median->iso4-6}).  However, if
	there is an isometric cycle of length $6$, then we can apply
	\cref{thm:planar-mg<->iso-cycles}~\ref{item:C6-Q3-proper}, to conclude
	that $G$ contains a cube $Q\lc{3}$, which is not possible by assumption.
	Hence, \Cref{item:iso-4} is satisfied.
	
	Conversely, assume that $G$ is a planar graph that satisfies
	\Cref{item:conect2,item:K23-free2,item:iso-4}.  Then, in particular, $G$
	satisfies the statements of \cref{thm:planar-mg<->iso-cycles}, which
	implies that $G$ is a median graph.  Now, assume for contradiction that
	$G$ contains a cube $Q\lc{3}$.  This $Q\lc{3}$ contains an isometric
	cycle $C\lc{6}$ w.r.t.\ $Q\lc{3}$.  By \Cref{lem:C6+Q3->conv-hull},
	$C\lc{6}$ is an isometric cycle in $G$; which is a contradiction to
	\Cref{item:iso-4}.  In summary, $G$ is a cube-free median graph.
\end{proof}

As a direct consequence of \Cref{lem:iso->in/out-C4} and since every square
in a median graph is isometric, we obtain
\begin{corollary}\label{lem:C4inGinGout}
	If $G$ is a median graph containing a square $C\subseteq G$, then all squares
	of $G$ are contained in the union of the squares contained in
	$\Gin{C,\pi}$ and $\Gout{C,\pi}$.
\end{corollary}
In other words, if we have two graphs $\Gin{C,\pi}$ and $\Gout{C,\pi}$ such
that $G = \Gin{C,\pi}\cd{C}{C}\Gout{C,\pi}$ results in a median graph, then
the only squares in $G$ are the ones contained in $\Gin{C,\pi}$ and
$\Gout{C,\pi}$.  \cref{lem:C4inGinGout} can be generalized further to
show  that $\{K\lc{3},K\lc{2,3}\}$-free planar graphs $G$ can be
characterized in terms of their subgraphs $\Gin{C,\pi}$ and $\Gout{C,\pi}$.
\begin{lemma}\label{lem:K2-K23-free-charact}
	Let $G$ be a $\pi$-embedded planar graph and $C\subseteq G$ be a
	square. Then, $G$ is $\{K\lc{3},K\lc{2,3}\}$-free if and only if
	$\Gin{C,\pi}$ and $\Gout{C,\pi}$ are $\{K\lc{3},K\lc{2,3}\}$-free.
\end{lemma}
\begin{proof}
	Let $G$ be a $\pi$-embedded planar graph, and let $C\subseteq G$ be a
	square.  If $G$ is $\{K\lc{3},K\lc{2,3}\}$-free, then every subgraph of
	$G$ is $\{K\lc{3},K\lc{2,3}\}$-free, and thus, $\Gin{C,\pi}$ and
	$\Gout{C,\pi}$ must be $\{K\lc{3},K\lc{2,3}\}$-free as well.
	
	The terms ``inside'' and ``outside'' in the following refer to the
	$(G,\pi)$-induced embedding of $C$. Consider two vertices $v,w \in V(G)$,
	where $v$ is outside of $C$ and $w$ is inside of $C$ (and thus, in
	particular, $v,w\notin V(C)$). We observe that $v$ and $w$ cannot be
	adjacent in $G$, since in the planar embedding $\pi$ an edge $\{v,w\}$
	would cross edges or vertices of $C$.
	
	Now, suppose that $\Gin{C,\pi}$ and $\Gout{C,\pi}$ are
	$\{K\lc{3},K\lc{2,3}\}$-free. Assume first, for contradiction, that $G$
	contains a subgraph $H\simeq K\lc{3}$. This subgraph can neither be
	located entirely in $\Gin{C,\pi}$ nor in $\Gout{C,\pi}$.  Hence, there
	are vertices $v,w\in V(H)$ such that $v$ is outside of $C$ and $w$ is
	inside of $C$. Since $H\simeq K_3$ it holds that
	$\{v,w\} \in E(H)\subseteq E(G)$; a contradiction.  Hence, $G$ must be
	$K\lc{3}$-free.
	
	Assume now, for contradiction, that $G$ contains a subgraph
	$H\simeq K\lc{2,3}$. Again, this subgraph can neither be located entirely
	in $\Gin{C,\pi}$ nor in $\Gout{C,\pi}$.  Hence, there are vertices
	$v,w\in V(H)$ such that $v$ is outside of $C$ and $w$ is inside of
	$C$. Since, as argued above, $\{v,w\} \in E(H)$ is not possible, we can
	conclude that $d\lc{H}(v,w)=2$, and thus, there must be (at least) two
	distinct paths $P\lc{H}(v,w)$ and $P'\lc{H}(v,w)$ of length $2$.  Let
	$V(P\lc{H}(v,w))=\{v,a,w\}$ and $V(P'\lc{H}(v,w))=\{v,b,w\}$.  Since
	$\{v,w\} \notin E(G)$ and $v$ is outside while $w$ is inside of that $C$,
	the only vertices that can be adjacent to $v$ and $w$ are vertices of
	$C$.  Hence, $a,b \in V(C)$. Note $d\lc{H}(v,w)=2$, and since $G$ is
	$K\lc{3}$-free, we can conclude that $d\lc{G}(a,b)=2$.  But then, the
	subgraph of $\Gout{C,\pi}$ induced by $V(C)\cup \{v\}$ contains a
	subgraph isomorphic to $K\lc{2,3}$; a contradiction.  Hence, $G$ is
	$\{K\lc{3},K\lc{2,3}\}$-free.
\end{proof}

Recall that a connected graph is either cyclic or a tree.  We are now
in the position to provide an additional characterization of planar 
graphs that contain squares and are median graphs.
\begin{theorem} \label{thm:if-inoutside->mg<->pmg} Let $G$ be a
	$\pi$-embedded planar graph that contains a square $C\subseteq G$.
	Then, $G$ is a median graph if and only if $\Gin{C,\pi}$ and $\Gout{C,\pi}$
	are median graphs.
\end{theorem}
\begin{proof}
	Let $G$ be a $\pi$-embedded planar graph, and let $C$ be a
	square of $G$.  By construction, $C\subseteq \Gin{C,\pi}$ and
	$C\subseteq \Gout{C,\pi}$.
	
	First, assume that $G$ is a median graph. Moreover, let $C'$ be an
	isometric cycle of $\Gin{C,\pi}$, and let $\mc{H}\lc{G}(C')$ be its
	convex hull w.r.t.\ $G$. Hence, \Cref{lem:in-out-C4-hull} implies that
	every $v' \in V(\mc{H}\lc{G}(C))$ lies almost-inside of this $C$ w.r.t.\
	the $(G,\pi)$-induced embedding of $G$. Thus, by definition of
	$\Gin{C,\pi}$, we conclude that $\mc{H}\lc{G}(C')\subseteq
	\Gin{C,\pi}$. Hence, since $C'\subseteq \Gin{C,\pi} \subseteq G$, we
	conclude that $\mc{H}\lc{G}(C')$ is also the convex hull of $C'$ w.r.t.\
	$\Gin{C,\pi}$. Since $G$ is a median graph,
	\cref{prop:convex-hull-charact} implies that $\mc{H}\lc{G}(C')$ is a
	hypercube. Thus, the convex hull of an arbitrary isometric cycle $C'$ in
	$\Gin{C,\pi}$ is a hypercube. Hence, \cref{prop:convex-hull-charact}
	implies that $\Gin{C,\pi}$ is a median graph. Analogously, one can show
	that $\Gout{C,\pi}$ is a median graph as well.
	
	Conversely, assume that $\Gin{C,\pi}$ and $\Gout{C,\pi}$ are median
	graphs.  \cref{lem:basic-median-props}~(1) implies that $\Gin{C,\pi}$ and
	$\Gout{C,\pi}$ are bipartite, and thus, they are $K_3$-free. This,
	together with \cref{lem:basic-median-props}~(2), implies that
	$\Gin{C,\pi}$ and $\Gout{C,\pi}$ are $\{K\lc{3},K\lc{2,3}\}$-free. By
	\Cref{lem:K2-K23-free-charact}, $G$ is $\{K\lc{3},K\lc{2,3}\}$-free. Now,
	let $C'$ be an isometric cycle of $G$. Hence, \Cref{lem:iso->in/out-C4}
	implies that all vertices $v \in V(C')$ lie almost-inside (resp.\ 
	almost-outside) of this $C$. First, assume that all vertices $v \in V(C')$
	lie almost-inside of the $C$. By definition of $\Gin{C,\pi}$ and by
	\Cref{lem:in-out-C4-hull}, we conclude that the convex hull
	$\mc{H}\lc{G}(C')$ of $C'$ w.r.t.\ $G$ is equal to the convex hull
	$\mc{H}\lc{\Gin{C,\pi}}(C')$ of $C'$ w.r.t.\ $\Gin{C,\pi}$. Since
	$\Gin{C,\pi}$ is a median graph, we conclude that
	$\mc{H}\lc{\Gin{C,\pi}}(C')$ is a hypercube. Thus, $\mc{H}\lc{G}(C')$ is
	a hypercube. Analogously, one can show that the convex hull
	$\mc{H}\lc{G}(C')$ of $C'$ w.r.t.\ $G$ is a hypercube if all vertices
	$v' \in V(C')$ are almost-outside of this $C$. Hence, in either case, the
	convex hull of any isometric cycle $C'$ of $G$ is a hypercube. Thus,
	\cref{prop:convex-hull-charact} implies that $G$ is a median graph.
\end{proof}

\section{Cubesquare-Graphs}
\label{sec:QS}

In this section, we establish a further characterization of planar median
graphs.  To this end, we provide a definition of an operator $\circledast$
to ``glue'' two graphs together.  This definition is motivated in part by 
\Cref{thm:if-inoutside->mg<->pmg}.
\begin{definition} \label{def:CD} Let $G$ and $H$ be two vertex-disjoint
	graphs with squares $C\subseteq G$ and $C'\subseteq H$. Let $\varphi$ be
	any isomorphism between the squares $C$ and $C'$.  Then, the composition
	$G\cd{C}{C'}H$ is obtained from $G$ and $H$ by identifying the vertices
	and edges of $C$ with their $\varphi$-images in $C'\subseteq H$.
\end{definition}
We will omit the explicit reference to $C$ and $C'$ in $\cd{C}{C'}$
whenever it is not needed. Note that $\circledast$ is not defined for graphs
that do not contain squares. Since the choice of
$\varphi$ will not play a role here, we suppress it in our
notation. 
\cref{fig:defCD} gives an illustrative example
of \cref{def:CD}.  There are eight different ways to define an isomorphism
on squares. Therefore, there are up to eight non-isomorphic graphs
$G\cd{C}{C'}H$ obtained by gluing together $G$ and $H$ at the same squares
with the help of different isomorphisms $\varphi$. 

\begin{figure}[t]
	\begin{center}
		\includegraphics[width=0.85\textwidth]{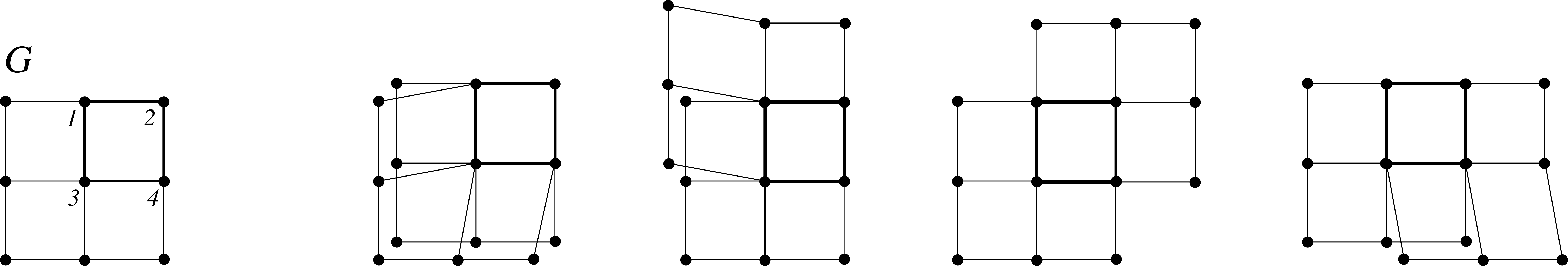}
	\end{center}
	\caption{There are three non-isomorphic graphs $G\cd{C}{C}G$ obtained by
		gluing together $G$ with a copy of itself at the square $C$. Shown are
		the four rotations of $C$ relative to its copy. By symmetry, the four
		rotations of the mirror image, i.e., mapping the vertex order
		$(1,3,4,2)$ to the order $(1,2,4,3)$ in the other copy, yields the same
		four configurations. Furthermore, the $2$nd and the $4$th case result
		in isomorphic graphs. Note that all graphs $G\cd{C}{C}G$ are planar. A
		generic planar embedding of the second case is provided by the drawing
		of the graph $G(2)$ in \cref{fig:exam-cubesquare}. However, only for
		the $3$rd case graph, the square $C$ remains a square-boundary and, in
		particular, the resulting graph is a square-graph.}
	\label{fig:defCD}
\end{figure}

It is easy to see that the square $C_4$ serves as a unique ``unit'' element,
that is, $G\circledast C_4 \simeq C_4\circledast G\simeq G$.  Moreover, the
operator~$\circledast$ is commutative, i.e.\
$G\cd{C}{C'} H \simeq H\cd{C'}{C} G$ for all graphs $G$ and $H$. 
However, it is not associative, since
$(G\lc{1}\cd{C}{C'} G\lc{2}) \cd{C''}{C'''} G\lc{3}$ can be well-defined,
but $G\lc{1}\cd{C}{C'} (G\lc{2} \cd{C''}{C'''} G\lc{3})$ is not; a case
that in particular happens when the square $C''$ is part of $G_1$ but not
of $G_2$; see \cref{fig:exam-cubesquare} for an example.

We will use the convention that $\circledast$-composition is read from left
to right, i.e.,
\begin{equation}
G\lc{1}\circledast G\lc{2}	\circledast G_3 \circledast \cdots
\circledast G_\ell \coloneqq
(\dots((G\lc{1}\circledast G\lc{2})	\circledast G_3) \circledast \cdots
\circledast G_{\ell-1}) \circledast G_\ell
\end{equation}
Setting
$G(i)\coloneqq G\lc{1}\circledast G\lc{2} \circledast G_3 \circledast
\cdots \circledast G_i$, we therefore have
$G(i+1) = G(i) \cd{C}{C'} G_{i+1}$, where $C$ is a square in $G(i)$ and
$C'$ is a square in $G_{i+1}$.
Note that, by
\cref{def:CD}, $G = \Gin{C,\pi}\cd{C}{C} \Gout{C,\pi}$.

In the following, we will consider the class of cubesquare-graphs as
defined below.  As we shall see later, a planar median graph is either a
tree or a cubesquare-graph.

\begin{definition} \label{def:QS-graphs} A \emph{cubesquare-graph} (or
	\emph{QS-graph} for short) is defined as follows:
	\begin{enumerate}
		\item[\textbf{(Q1)}] Every cube $Q\lc{3}$ and every cyclic square-graph
		is a QS-graph, called \emph{basic} QS-graph.
		\item[\textbf{(Q2)}] The ordered composition $G(\ell)$ of basic QS-graphs
		$G_i$, $1\le i\le \ell$ is a QS-graph, where $G(\ell)$ is defined
		recursively as $G(1)=G_1$ and $G(i)=G(i-1)\cd{C_{i-1}}{C_i} G_{i}$,
		$2\le i\le\ell$ using \emph{square-boundaries} $C_{i-1}$ in $G(i-1)$
		and $C_{i}$ in $G_{i}$.
	\end{enumerate}
\end{definition}
In other words, every QS-graph can be obtained from a cube or a
square-graph by iteratively replacing boundaries (w.r.t.\ some embedding)
that are $4$-cycles by cubes or square-graphs.  We emphasize that, in
contrast to \cref{def:CD}, the squares chosen in the construction of
QS-graphs are not arbitrary but must be square-boundaries for some planar
embedding in each iteration.  We shall see below that this construction is
always possible since each partial composition $G(i)$ and each basic
QS-graph contains a square-boundary.  An illustrative example of QS-graphs
is given in \cref{fig:exam-cubesquare}.

\begin{figure}[t]
	\begin{center}
		\includegraphics[width=0.7\textwidth]{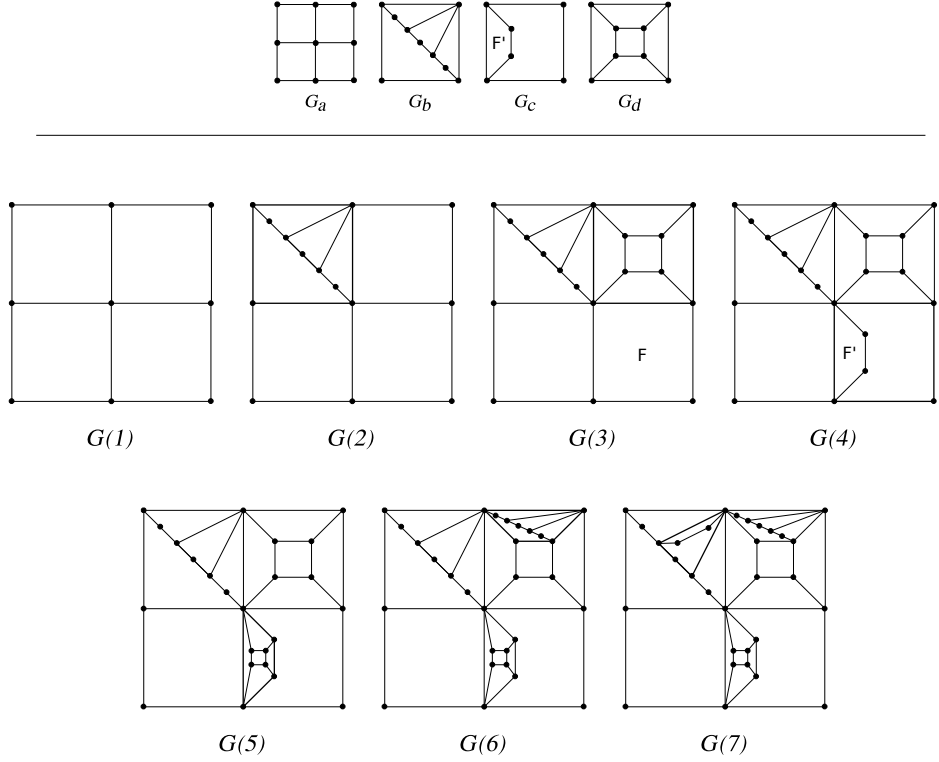}
	\end{center}
	\caption{In the upper part, four graphs $G_a,G_b,G_c$ and $G_d$ are
		shown.  Here, $G_a\simeq G_b$ is a square-graph and
		$G_d$ is a cube. Moreover, $G_c$ is a square-graph since there is an
		alternative planar embedding $\sigma$ such that all vertices become
		outer vertices, and thus, that each inner face is bounded by a square.
		By \cref{def:QS-graphs}, $G_a$ to $G_d$ are (basic) QS-graphs.
		Moreover, $G(1)=G_a$, and thus, $G(1)$ is a QS-graph.  We have
		$G(7) = (((((G_a \circledast G_b)\circledast G_d)\circledast
		G_c)\circledast G_d)\circledast G_b)\circledast G_c$.  Hence, every
		$G(i)$ with $i \in \{2,\ldots,7\}$ is precisely the graph
		$G(i-1) \circledast G_x$, where $G_x$ is the appropriate QS-graph with
		$x\in \{b,c,d\}$ together with its shown planar embedding.  Therefore,
		every $G(i)$ is a QS-graph. \newline This example also shows that
		$\circledast$ is not associative, in general.  To see this, consider
		the QS-graph $G(3) = (G_a \cd{C_a}{C_b} G_b)\cd{C'_a}{C_d} G_d$ where
		$C_a$ and $C'_a$ are the two squares in $G_a$ that are ``identified''
		with the squares $C_b$ and $C_d$ that form the outer boundary in $G_b$
		and $G_d$, respectively.  Hence, we cannot write $G(3)$ as
		$G_a \cd{C_a}{C_b} (G_b\cd{C'_a}{C_d} G_d)$ since the square $C'_a$
		does not exist in $G_b$.  In some cases, however,
		associativity is given. By way of example, consider
		$G(5) = (G(3)\cd{C}{C_c} G_c) \cd{C_c}{C_d} G_d$, where $C$ is the
		square in $G(3)$ that bounds $F$, $C_c$ the square that bounds $F'$ and
		$C_d$ the outer-boundary of $G_d$.  It is easy to see that
		$G(5) \simeq G(3)\cd{C}{C_c} (G_c \cd{C_c}{C_d} G_d)$, since the these
		constructions ``overlap'' on the cycle $C_c$. }
	\label{fig:exam-cubesquare}
\end{figure}

In the following, we will make frequent use of the planar embeddings $\rho$
and $\rho_C$ of cubes (cf.\ \cref{obs:squaregraph-cube-embedding}) and the
embeddings $\sigma$ and $\sigma_C$ of square-graphs (cf.\ \cref{def:s-g}
and \cref{obs:squaregraph-cube-embedding}).
\begin{lemma}\label{lem:QS-wellDef}
	QS-graphs are well-defined and planar graphs that satisfy $1$-FS.
\end{lemma}
\begin{proof}
	Recall that for $G = G_1\cd{C}{C}G_2$, the square $C$ must be a
	square-boundary in $G_1$ and $G_2$ but not necessarily in $G$, see
	\cref{fig:defCD}.  In order to show that QS-graphs are well-defined, we
	must, in particular, show that in each step of creating a new QS-graph at
	least one square-boundary remains which allows us to add another QS-graph
	(cf.\ \textbf{(Q2)}).  Hence, we must show that every QS-graph satisfies
	$1$-FS.  This, in particular, implies that QS-graphs must be planar.
	
	Let us first consider basic QS-graphs. By
	\Cref{lem:square-faces-squaregraph} and \Cref{obs:Cube-Embedding}, every
	square-graph and cube satisfies $1$-FS.  In particular, if a basic
	QS-graph contains at least two squares, then \Cref{obs:Cube-Embedding}
	and \Cref{lem:square-faces-squaregraph} imply that it must satisfy
	$2$-FS.  Thus, every basic QS-graphs satisfies $1$-FS and every basic
	QS-graph containing at least two squares satisfies $2$-FS.
	
	We proceed now by induction on the number $\ell$ of factors to show that
	the ordered composition $\circledast$ of $\ell$ basic QS-graphs is
	well-defined and satisfies $1$-FS which, in particular, implies that we
	obtain a planar QS-graph. The base case are the basic QS-graphs. Assume
	that $G\lc{1}\circledast G\lc{2} \circledast \cdots \circledast G_i$
	is well-defined and results in a planar graph that satisfies $1$-FS for
	all $1\leq i<k$. Consider now a product of $k$ basic QS-graphs
	$G\lc{1}\circledast G\lc{2} \circledast \cdots \circledast G_k$. Set
	$H\coloneqq G(k-1)$ and $H'\coloneqq G_k$. We show first,
	\begin{owndesc}
		\item[\textnormal{\emph{Claim 1:}}] 
		$H\cd{C}{C'} H'$ is well-defined and a planar graph.
		
		\emph{Proof of Claim 1.}  By assumption, $H$ and $H'$ satisfy
		$1$-FS. Let $C$ be a square-boundary of $H$ and $C'$ a square-boundary
		of $H'$ w.r.t.\ some planar embedding of $H$, resp., $H'$.  Now, we can
		use the embeddings $\pi_{C'}\in \{\sigma_{C'},\rho_{C'}\}$ for $H'$
		depending on whether $H'$ is a cube or a cyclic square-graph such that
		the square $C'$ is the outer boundary of $H'$.  Since $C$ is a
		square-boundary in $H$, there is a planar embedding $\pi_C$ of $H$ such
		that $C$ is an inner boundary w.r.t.\ $\pi_C$ (cf.\
		\cref{obs:cycle-bound}).  The outer boundary $C'$ of $H'$ intersects
		$H\subseteq H\cd{C}{C'} H'$ only in the $4$ vertices of the chosen
		square $C$ by definition of $H\cd{C}{C'} H'$.  Now, consider the
		embedding $\kappa(\pi_{C},\pi_{C'})$ of $H\cd{C}{C'} H'$.  It consists
		of the drawing of $H\cd{C}{C'} H'$ based on $\pi\lc{C}$ together with
		the ``scaled'' planar drawing $\pi\lc{C'}$ such that $H'$ intersects
		$H$ only in the vertices contained in $C$ and the remaining vertices of
		$H'$ are placed inside of the inner face of $H$ bounded by $C$. Thus it
		yields a planar embedding of $H\cd{C}{C'} H'$.  In summary,
		$H\cd{C}{C'} H'$ is well-defined and results in a planar graph.  \hfill
		$\diamond$
	\end{owndesc}

	\begin{owndesc}
		\item[\textnormal{\emph{Claim 2:}}] $H\cd{C}{C'} H'$ satisfies $1$-FS.
		
		\emph{Proof of Claim 2.}  We will make frequent
		use of the planar embeddings $\pi_{C}$ of $H$, $\pi_{C'}$ of $H'$ and
		$\kappa(\pi_{C},\pi_{C'})$ of $H\cd{C}{C'} H'$ as specified in the
		proof of Claim 1.
		
		First, assume that $H'$ contains only one square. Since $H'$ is a basic
		QS-graph, it must therefore be square-graph and thus, $H'$ is
		isomorphic to a $C_4$ to which possibly a couple of trees are
		attached. Let $\tilde C$ be the square in $H\cd{C}{C'}H'$ that refers to
		the two identified cycles $C$ and $C'$ via the chosen subgraph
		isomorphism.  By construction of $\kappa(\pi_{C},\pi_{C'})$, $\tilde C$
		together with these possible attached tree forms an inner boundary in
		$H\cd{C}{C'} H'$ w.r.t.\ $\kappa(\pi_{C},\pi_{C'})$.  Hence, every tree
		that is attached to a vertex $v$ in $\tilde C$ can safely be re-located
		in some face of $H\cd{C}{C'} H'$ that is incident to $v$ w.r.t.\
		$\kappa(\pi_{C},\pi_{C'})$.  In this way, we obtain a new planar
		embedding of $H\cd{C}{C'} H'$ such that $\tilde C$ is the boundary of
		an inner face and thus, $H\cd{C}{C'} H'$ satisfies $1$-FS.
		
		Now, assume that $H'$ contains more than one square. As argued
		above, $H'$ together with  its planar embedding
		$\pi_{C'}\in \{\sigma_{C'},\rho_{C'}\}$ satisfies  $2$-FS
		w.r.t.\ $\pi_{C'}$ and thus in $H'$ there are two faces bounded by a
		square w.r.t.\ $\pi_{C'}$. Since $C'$ is an outer boundary of $H'$
		w.r.t.\ $\pi_{C'}$ the other face that is bounded by a square $C''$
		must be an inner face. It is straightforward to see that $C''$ still
		bounds an inner face in $H\cd{C}{C'} H'$ w.r.t.\
		$\kappa(\pi_{C},\pi_{C'})$.  Hence, $H\cd{C}{C'} H'$ satisfies $1$-FS.
		\hfill $\diamond$
	\end{owndesc}
	In particular, $H\cd{C}{C'} H'$ is planar and contains the required
	square-boundary. Thus,
	$G\lc{1}\circledast G\lc{2} \circledast \cdots \circledast G_k$ is a
	well-defined planar graph for all $k$.
\end{proof}

\begin{remark}\label{rem:speci-planar}
	For a QS-graph $G(i)\cd{C}{C'} G_{i+1}$, we will use the notation
	$\pi\lc{C}$ as well as $\pi\lc{C'}\in \{\rho_{C'}, \sigma_{C'}\}$ and
	$\kappa(\pi_C, \pi_{C'})$ for the planar embedding of $G(i),G_{i+1}$ and
	$G(i)\cd{C}{C'} G_{i+1}$, respectively, as specified in the proof of
	\Cref{lem:QS-wellDef}.
\end{remark}

\begin{lemma}\label{lem:Qthree->pmg}
	Every QS-graph is a planar median graph.
\end{lemma}
\begin{proof}
	We show now, by induction on the number of factors, that every QS-graph
	is a median graph.  As base case, we have a basic QS-graph, i.e., either
	a cube or a cyclic square graph. These are planar and
	\Cref{lem:square-graph:basic-proper} implies that they are median graphs.
	Now, let
	$G = G\lc{1}\circledast G\lc{2} \circledast \cdots \circledast G_k =
	G(k-1)\cd{C}{C'} G_k$ be the ordered composition of $k$ basic QS-graphs
	and assume that the ordered composition of $i<k$ basic QS-graphs is a
	median graph.  Since $G = G(k-1)\cd{C}{C'} G_k$ is planar by
	\Cref{lem:QS-wellDef}, we can use the planar embedding
	$\pi \coloneqq \kappa(\pi_C, \pi_{C'})$ of $G$ (cf.\
	\Cref{rem:speci-planar}).  It is straightforward to verify that
	$G(k-1)= \Gout{C, \pi}$ and $G_k = \Gin{C,\pi}$. Thus,
	\cref{thm:if-inoutside->mg<->pmg} implies that $G$ is a median graph.
\end{proof}

We now want to consider the converse of \Cref{lem:Qthree->pmg}. 
We begin with some observations.

\begin{lemma}\label{lem:forb-planarE}
	Let $G$ be a book or a suspended cogwheel.  Then, for any planar
	embedding $\pi$ of $G$, there is a square $C^* \subseteq G$ such that
	$\Gin{C^*,\pi}\neq G$ and $\Gout{C^*,\pi}\neq G$.
\end{lemma}
\begin{proof}
	First, let $G = K_2\Box K_{1,3}$ be a book.  Assume that $V(K_2)=\{0,1\}$
	and $V(K_{1,3}) = \{0,1,2,3\}$ where $0$ is the unique vertex adjacent to
	the remaining ones.  By definition of the Cartesian product,
	$V(G) = \{00,10,01,11,02,12,03,13\}$ and we have exactly three squares
	$C,C',C''$ in $G$ that consist of the vertices $V(C) =\{00,10,02,12\}$,
	$V(C') =\{00,10,01,11\}$, and $V(C'') =\{00,10,03,13\}$.  Below, the
	terms ``inside'' and ``outside'' of some subgraph $G'\subseteq G$ refer
	to the $(G,\pi)$-induced embedding of $G'$.  Now, consider the square $C$
	and an arbitrary embedding $\pi$ of $G$.  We have to examine the cases
	that $k_1$ vertices are inside $C$ and $k_2$ are outside of $C$ where
	$k_1+k_2=4$, the number of remaining vertices in
	$V(G)\setminus V(C)$.  Hence, $k_1\in \{0,1,2,3,4\}$.
	
	Let us start with $k_1=2$ and thus, $k_2=2$.  Assume w.l.o.g.\ that $01$
	is inside of $C$. In this case, \cref{lem:in-out-C4} implies that the
	second vertex inside of $C$ must be vertex $11$. Hence, the two vertices
	outside of $C$ are $03$ and $13$. Now one readily observes that
	$\Gin{C,\pi}\neq G$ and $\Gout{C,\pi}\neq G$.  The latter reasoning implies
	that the case $k_1=1$ cannot occur, since if one vertex is inside $C$
	there must also be a second vertex inside $C$.  By symmetry, this also
	excludes the case $k_2=1$ and thus, $k_1=3$.
	
	Thus, we are left with the case $k_1\in \{0,4\}$. Let $k_1=4$.  Consider
	the square $C'$. By similar arguments as above, there are either $0$, $2$
	or $4$ vertices inside of $C'$.  However, the latter case cannot occur,
	since all $k_1=4$ remaining vertices are inside $C$ and thus, there must
	be vertices outside $C'$.  If there are $2$ vertices inside of $C'$, then
	the vertices inside $C'$ must be $03$ and $13$ since $C'$ is
	almost-inside of $C$. Now, one easily verifies that $\Gin{C',\pi}\neq G$
	and $\Gout{C',\pi}\neq G$.  If there are no vertices inside of $C'$, then
	all vertices $02,12,03,13$ must be outside of $C'$. Since $03,13$ are
	inside of $C$, we can conclude that $\Gin{C'',\pi}\neq G$ and
	$\Gout{C'',\pi}\neq G$.  By symmetry, the case $k_2=4$ and thus, $k_1=0$
	is shown.
	
	In summary, for all possible cases we found a square $C^* \subseteq G$ in
	the book $G$ such that $\Gin{C^*,\pi}\neq G$ and $\Gout{C^*,\pi}\neq G$.
	
	Now, let $G\simeq M_n^*$ be a suspended cogwheel.  Let $x$ be the vertex
	that is adjacent to the central vertex $c$ of the underlying cogwheel
	$M_n$. Moreover, denote with $c_1,\dots,c_n$ the vertices of the cycle
	$C_n\subseteq M_n^*$ that are distinct from $x$ and adjacent to $c$. We
	assume that the edges of this cycle $C_n$ are $\{c_n,c_1\}$ and
	$\{c_i,c_{i+1}\}$, $1\leq i\leq n-1$.  In addition, let $c_{k}$ with $k$
	being even be the vertices that are adjacent to $c$. In what follows, all
	indices $j,i,i+1,i+2,\dots$ are taken w.r.t.\ $(\mathrm{mod}\ n)$.
	
	Let $\pi$ be an arbitrary embedding of $G$.  In what follows, the terms
	``inside'' and ``outside'' of some subgraph $G'\subseteq G$ refer to the
	$(G,\pi)$-induced embedding of $G'$.  It is easy to see that the vertex
	$x$ must be located in one of the faces that is bounded by a subgraph
	$G'\subseteq G$ such that $G'$ contains two vertices $c_i$ and $c_{i+2}$
	with $i$ being even.  Thus, assume that $x$ is in a face bounded by
	$G'\subseteq G$ such that $c_i, c_{i+2}\in V(G')$, $i$ even.  We continue
	by showing that, in this case, the square $C+x$ induced by the vertices
	$x, c, c_i, c_{i+1}, c_{i+2}$ must be a boundary in $G$.  Assume, for
	contradiction, that this is not the case.  Hence, one $c_{j}$ with
	$j\notin \{i,i+1,i+2\}$ must be contained inside of $C$.  But then also
	$c_{j+1}$ must be contained inside of $C$ as otherwise the edge
	$\{c_j,c_{j+1}\}$ would cross one of the edges or vertices of $C$
	w.r.t.\ the planar drawing $\pi$ of $G$. Repeating the latter argument
	shows that all vertices $c_{j}$ with $j\notin \{i,i+1,i+2\}$ must be
	located inside of $C$. In this case, however, $x$ is located in a face
	that is bounded by some $G''\subseteq G$ that contains the vertices
	$c_j, c_{j+2}\in V(G')$ where $j$ is even and where at least one of $c_j$
	and $c_{j+2}$ is distinct from $c_i$ or $c_{i+1}$; a contradiction.
	Hence, $C+x$ must be a boundary in $G$.
	
	Finally, observe that either $\Gin{C,\pi} = C+x$ or $\Gout{C,\pi} = C+x$
	and thus, either $\Gout{C,\pi} = G-x$ or $\Gin{C,\pi} = G-x$.  Hence, we
	found the square $C^* =C$ such that $\Gin{C^*,\pi} \neq G$ and
	$\Gout{C^*,\pi}\neq G$.
\end{proof}

\begin{lemma}\label{lem:decomp}
	If $G$ is a $\pi$-embedded cyclic planar median graph that is not a basic
	QS-graph,  then there is a square $C^* \subseteq G$ such that
	$\Gin{C^*,\pi} \neq G$ and $\Gout{C^*,\pi} \neq G$.
\end{lemma}
\begin{proof}
	Let $G$ be a $\pi$-embedded planar median graph that is not a basic
	QS-graph.  \cref{prop:square-from-median} implies that $G$ must contain a
	cube, a book or a suspended cogwheel.  Let $H\subseteq G$ be such a
	forbidden subgraph.
	
	First, assume that $H$ is a cube.  Note, every face of $H$ must be
	bounded by squares (cf.\ \cref{obs:Cube-Embedding}).  Since $H\subseteq G$
	and $G$ is not a cube, there must be a vertex $v\in V(G)\setminus V(H)$
	that lies in a face of $H$ that is bounded by a square $C^*\subseteq H$
	in $H$ (w.r.t.\ the $(G,\pi)$-induced embedding of $H$).  Note, $C^*$ is
	not necessarily a boundary in $G$ but, of course, a subgraph of $G$.  Let
	$v' \in V(H)\setminus (V(C^*)\cup\{v\})$. If $v$ lies in the outer face
	of $H$ w.r.t.\ the $(G,\pi)$-induced embedding, then
	$\Gin{C^*,\pi}\subseteq G-v\subsetneq G$ and
	$\Gout{C^*,\pi}\subseteq G-v'\subsetneq G$.  Otherwise, if $v$ lies in an
	inner face of $H$ w.r.t.\ the $(G,\pi)$-induced embedding, then
	$\Gout{C^*,\pi}\subseteq G-v\subsetneq G$ and
	$\Gin{C^*,\pi}\subseteq G-v'\subsetneq G$.  In either case, there is a
	square $C^*$ such that $\Gin{C^*,\pi} \neq G$ and $\Gout{C^*,\pi} \neq G$.
	
	If $H$ is a book or a suspended cogwheel, then we can apply
	\Cref{lem:forb-planarE} to conclude that there is a square
	$C^*\subseteq H$ such that $\Gin{C^*,\pi}\neq H$ and
	$\Gout{C^*,\pi}\neq H$ for every planar embedding of $H$ and thus, in
	particular, for the $(G,\pi)$-induced embedding of $H$. The latter
	immediately implies that $\Gin{C^*,\pi} \neq G$ and
	$\Gout{C^*,\pi} \neq G$.
\end{proof}

\begin{lemma}\label{lem:pmg->Qthree}
	Every cyclic planar median graph is a QS-graph.
\end{lemma}
\begin{proof}
	Since squares that are possibly amalgamated with trees are square-graphs
	and thus, median graphs, for every integer $n\geq 4$ there is a cyclic
	planar median graph on $n\ge 4$ vertices. Thus, we can proceed by
	induction on $|V(G)|$. The square is the only cyclic median graph with
	$n=4$ vertices. By definition, it is also a (basic) QS-graph, and thus
	serves as base case.
	
	For the induction step consider a $\pi$-embedded planar median graph $G$
	with $n=|V(G)|>4$ vertices and assume that every planar median graph $G'$
	with $|V(G')|<|V(G)|$ vertices is a QS-graph. If $G$ is a cube or a
	square-graph then $G$ is a QS-graph and we are done.  Hence, assume that
	$G$ is neither a cube nor a square-graph and let $\pi$ be a planar
	embedding of $G$. By \Cref{lem:decomp}, there is a square
	$C^* \subseteq G$ such that $\Gin{C^*,\pi}\neq G$ and
	$\Gout{C^*,\pi}\neq G$.  As shown in the proof of \Cref{lem:decomp}, we
	can find such a square $C^*$ by taking a forbidden subgraph $H$, i.e., a
	book, a cube or suspended cogwheel, and a particular square
	$C^*\subseteq H$.  We may assume w.l.o.g.\ that $\Gin{C^*,\pi}$ is either
	a cube or 
	does not contain a book, a cube or a suspended
	cogwheel as a subgraph. Otherwise, we could
	iteratively replace $H$ by such a forbidden subgraph
	$\tilde H\subseteq \Gin{C^*,\pi}$, and replace $C^*$ by a square
	$\tilde C^*$ of $\tilde H$ with the property that
	$\Gin{\tilde C^*,\pi}\neq G$ and $\Gout{\tilde C^*,\pi}\neq G$, until we
	eventually obtain a forbidden subgraph $H$ and a square $C^*$ such that
	$\Gin{C^*,\pi}$ is either a cube or, otherwise, does not any longer
	contain a book, a cube or a suspended cogwheel.
	
	By \Cref{thm:if-inoutside->mg<->pmg}, $\Gin{C^*,\pi}$ and
	$\Gout{C^*,\pi}$ are median graphs. If $\Gin{C^*,\pi}$ is not a cube,
	then \Cref{prop:square-from-median} and the fact that $\Gin{C^*,\pi}$
	does not contain a cube, book or a suspended cogwheel, implies that
	$\Gin{C^*,\pi}$ is a square-graph.  In either case, $\Gin{C^*,\pi}$ is a
	basic QS-graph.  Moreover, since $\Gout{C^*,\pi}$ is a median graph with
	$|V(\Gout{C^*,\pi})|<|V(G)|=n$, induction hypothesis implies that
	$\Gout{C^*,\pi}$ is a QS-graph.  Hence, $\Gout{C^*,\pi}$ has an ordered
	composition
	$\Gout{C^*,\pi} = (\dots (G_1\circledast G_2) \dots )\circledast G_k$ of
	basic QS-graphs.  Therefore, the ordered composition
	$(\dots (G_1\circledast G_2) \dots )\circledast G_k) \cd{C^*}{C^* }
	\Gin{C^*,\pi}$ is well-defined and yields a QS-graph that is identical to
	$G$.
\end{proof}

As an immediate consequence of \Cref{lem:pmg->Qthree,lem:Qthree->pmg} we
obtain
\begin{theorem}\label{thm:mpg<->Qthree}
	A graph $G$ is a planar median graph if and only if $G$ is a QS-graph or
	a tree.
\end{theorem}
\Cref{thm:mpg<->Qthree}, together with \Cref{thm:if-inoutside->mg<->pmg},
furthermore implies the following:
\begin{theorem} \label{thm:if-inoutside->QSthree<->cubesq} 
	Let $G$ be a $\pi$-embedded planar
	graph and  $C\subseteq G$ be a square.
	Then, $G$ is a QS-graph if and only if $\Gin{C,\pi}$
	and $\Gout{C,\pi}$ are QS-graphs.
\end{theorem}

\begin{corollary}\label{prop:NrSquares}
	A planar median graph with $n$ vertices contains $\mc{O}(n)$ squares.
\end{corollary}
\begin{proof}
	Let $G$ be a planar median graph. Hence, it is a QS-graph with
	composition $G = G_1 \circledast \dots \circledast G_k$ of basic
	QS-graphs. Let $n_i$ be the number of vertices in $G_i$, $1\leq i\leq k$.
	If a factor $G_i=(V_i,E_i)$ is a square-graph, then it contains
	$\mc{O}(n_i)$ squares (cf.\ \cite[Cor.\ 5]{klavzar1998euler}) and, if it
	is a cube it contains $6 < n_i=8$ squares.  Hence, each factor $G_i$ adds
	$\mc{O}(n_i)$ squares to $G$ and therefore, $G$ has
	$\mc{O}(\sum_i n_i)=\mc{O}(n)$ squares.
\end{proof}

\Cref{lem:basic-median-props,prop:square-from-median,lem:forb-planarE} can
be used to obtain the following interesting result.
\begin{proposition} \label{prop:pmg:convHull-sqG} Let $G$ be a planar
	median graph.  Then, the convex hull of each boundary is a square-graph.
\end{proposition}
\begin{proof}
	Let $G$ be a planar median graph together with some planar embedding
	$\pi$. Moreover, let $G'$ be some arbitrary boundary in $G$. Then,
	\Cref{lem:basic-median-props}~(4a) implies that the convex hull
	$\mc{H}(G')$ of $G'$ in $G$ is a median graph.
	
	We continue with showing that the convex hull $\mc{H}(G')$ is a
	square-graph. To this end, assume for contradiction that there is a
	subgraph $H\subseteq \mc{H}(G')$ that is isomorphic to a cube, a book, or
	a suspended cogwheel.
	
	If $H$ is isomorphic to a book or a suspended cogwheel, then
	\cref{lem:forb-planarE} implies that there is a square
	$C^*\subseteq H\subseteq \mc{H}(G')$ such that there is a vertex
	$v \in V(H)$ inside and a vertex $w \in V(H)$ outside of $C^*$ (w.r.t.\
	its $(G,\pi)$-induced embedding).  Since $G'$ is a boundary in $G$, it
	has to be almost-inside or almost-outside of $C^*$.  The latter two
	statements together with \cref{lem:in-out-C4-hull} imply that
	$v \in V(H)\subseteq V(\mc{H}(G'))$ or
	$w \in V(H)\subseteq V(\mc{H}(G'))$ cannot be a part of $\mc{H}(G')$; a
	contradiction.  Hence, $H$ cannot be a book or a suspended cogwheel, and
	thus $H$ must be a cube.  However, all vertices $v \in V(G')$ have to be
	in some face of that cube $Q\lc{3}$ (w.r.t.\ its $(G,\pi)$-induced
	embedding), which has to be bounded, in
	particular, by a square $C^*$ (cf.\ \cref{obs:Cube-Embedding}).  By
	\cref{lem:in-out-C4-hull}, none of the vertices
	$v' \in V(Q\lc{3})\setminus V(C^*)\ne \emptyset$ can be part of
	$\mc{H}(G')$; a contradiction.  Thus, $H$ cannot be a subgraph of
	$\mc{H}(G')$.
	
	This together with $\mc{H}(G')$ being a median graph and
	\Cref{prop:square-from-median} implies that $\mc{H}(G')$ is a
	square-graph.
\end{proof}
As an illustration of \Cref{prop:pmg:convHull-sqG}, we refer to
\cref{fig:exam-cubesquare}. Consider the graphs $G(1)$ and $G(7)$ with its
outer boundary; the $8$-cycle $C$.  Here, $\mc{H}(C)\simeq G(1)$ is a
square graph.  Note, however, not all boundaries are necessarily cycles
(cf.\ Fig.\ \ref{fig:planarDef}).

A graph $G$ is an \emph{amalgam} of two induced subgraphs $G_1$ and $G_2$
if their union is $G$ and their intersection is non-empty.  Every finite
median graph is obtained from a collection of hypercubes by successive
amalgamations of convex subgraphs \cite{isbell1980median,vV:84}.  Every
pseudo-median graph can be built up by successive amalgamations along
so-called gated subgraphs of certain Cartesian products of wheels, snakes
(i.e., path-like 2-trees), and complete graphs minus matchings
\cite{BH:91}. We now explain how our results also fit into this framework.  
To this end, we define square-boundary amalgamations as follows.  A graph $G$ is a
\emph{square-boundary amalgam (w.r.t. $C$)} of two induced subgraphs $G_1$
and $G_2$, if $G$ is an amalgam of $G_1$ and $G_1$ and the intersection
$G_1\cap G_2 \coloneqq C$ is a square-boundary of both $G_1$ and $G_2$.
\begin{obs} \label{obs:amal} $G$ is a square-boundary amalgam of two
	induced subgraphs $G_1$ and $G_2$ w.r.t.\ $C$ if and only of
	$G = G_1\cd{C}{C} G_2$.
\end{obs}

\Cref{thm:mpg<->Qthree} and \Cref{obs:amal} can be used to show the following
\begin{theorem} \label{thm:amalgam} A graph is a planar median graph if and
	only if it can be obtained from cubes and square-graphs by a sequence of
	square-boundary amalgamations.
\end{theorem}
\begin{proof}
	Let $G$ be a planar median graph. By \Cref{thm:mpg<->Qthree}, $G$ is a
	tree or a QS-graph. If $G$ is a tree, then it is a square-graph.
	Otherwise, $G$ is cyclic and thus, by definition, there is an ordered
	composition
	$G=G\lc{1}\circledast G\lc{2} \circledast \cdots \circledast G_k$ of
	basic QS-graphs, that is, cubes or cyclic square-graphs.  In particular,
	the subgraphs $G_i\subseteq G$, $1\leq i\leq k$ are induced.  Hence, $G$
	is obtained from cubes and square-graphs by a sequence of square-boundary
	amalgamations. Conversely, if $G$ is obtained from cubes and
	square-graphs by a sequence of square-boundary amalgamations, the graph
	$G$ must be a tree or a QS-graph and thus, by \Cref{thm:mpg<->Qthree}, a
	planar median graph.
\end{proof}

\section{Fast Decomposition of Planar Median Graphs into an Ordered
	Sequence of Basic QS-graphs}
\label{sec:algo}

\cref{lem:decomp} immediately implies a recursive strategy to determine an
ordered composition of QS-graphs of a given planar median
graph. Importantly, it is not necessary to find forbidden subgraphs as in
the proof of \cref{lem:pmg->Qthree}, which we used in this proof to
properly apply the induction step.  First, we test if $G$ is a planar
median graph and, in the affirmative case, compute a planar embedding $\pi$
of $G$ and continue.  If $G$ is square-graph or a cube, we are done.
Otherwise, $G$ is not a basic QS-graph and thus, by \cref{lem:decomp},
there is a square $C \subseteq G$ such that $\Gin{C,\pi}\neq G$ and
$\Gout{C,\pi}\neq G$.  There are two cases, either (i) both $\Gin{C,\pi}$
and $\Gout{C,\pi}$ are basic QS-graphs or (ii) at least one of them is
not. In Case (i), we are done, since we found a decomposition of
$G = \Gin{C,\pi}\cd{C}{C} \Gout{C,\pi}$ into basic QS-graphs.  In Case
(ii), we recurse on the non-basic QS graph $\Gin{C,\pi}$ or $\Gout{C,\pi}$,
resp., and repeat the latter until all such squares have been examined.  In
this recursion, we must, however, determine for all remaining squares $C$
after we found a basic QS-graph $H'$ if $\Hin{C,\pi}\neq H$ and
$\Hout{C,\pi}\neq H$ where $H$ is obtained from $G$ by removing $H'$ except
for the square $C$, and, in particular, keep track of the order of the
chosen factors to obtain an ordered composition of the input graph.

To address these issues, we will design a 
non-recursive algorithm instead.  To this end, we will use a partial order
on the set of all squares of $G$ that is defined in term of
``almost-inside'' w.r.t.\ $(G,\pi)$ induced embedding
\begin{definition}
	A square $C'\subseteq G$ is \emph{almost-inside} a square $C\subseteq G$,
	in symbols $C\preceq_{G,\pi} C'$, if all vertices of $C'$ are
	almost-inside $C$ w.r.t.\ the $(G,\pi)$-induced embedding of $C$.  In
	particular, we write $C\prec_{G,\pi} C'$ if $C\preceq_{G,\pi} C'$ and
	$C\neq C'$.
\end{definition}
In the following let $\mathcal{S}(G)$ denote the set of all squares
contained in $G$.
\begin{lemma}
	For every
	$\pi$-embedded planar graph $G$, $(\mathcal{S}(G),\preceq_{G,\pi})$ is a partially ordered set.
\end{lemma}
\begin{proof}
	In the following, the term ``(almost-)inside'' and ``outside'' refers to
	the $(G,\pi)$-induced embedding.  By definition, $\preceq_{G,\pi}$ is
	reflexive.  Moreover, if $C \preceq_{G,\pi} C' \preceq_{G,\pi} C''$, then
	all vertices of $C$ are inside or part of $C'$. The same applies for $C'$
	and $C''$. Now, it is easy to see that $\preceq_{G,\pi}$ is transitive
	(i.e., $C \preceq_{G,\pi} C''$).
	We continue with showing that $\preceq_{G,\pi}$ is anti-symmetric.  To
	this end, let $C,C'\in \mathcal{S}(G)$ such that $C\preceq_{G,\pi}C'$ and
	$C'\preceq_{G,\pi}C$.  Assume, for contradiction, that $C\neq C'$.  Since
	$C\preceq_{G,\pi}C'$ and $C\neq C'$, at least one vertex of $C$ must be
	located inside of $C'$ while all other vertices of $C$ are almost-inside
	of $C'$. But then, at least one vertex of $C'$ must be outside of $C$ and
	thus, $C'\not\preceq_{G,\pi}C$; a contradiction.  Therefore,
	$\preceq_{G,\pi}$ is anti-symmetric.  In summary,
	$(\mathcal{S}(G),\preceq_{G,\pi})$ is a partially ordered set.
\end{proof}

Next, we consider a condition for the nesting of squares in planar median
graphs.
\begin{lemma}\label{lem:1in}
	Let $G$ be a $\pi$-embedded planar median graph and let
	$C,C'\in \mathcal S(G)$.  If there is a vertex $v\in V(C)$ such that $v$
	is inside of $C'$ w.r.t.\ $(G,\pi)$-induced embedding, then
	$C\prec_{G,\pi} C'$, i.e., $C$ is almost-inside $C'$ w.r.t.\
	$(G,\pi)$-induced embedding.
\end{lemma}
\begin{proof}
	In the following, the term ``(almost-)inside'' and ``outside'' refers to
	the $(G,\pi)$-induced embedding. Let $C,C'\in \mathcal S(G)$ and suppose
	that there is a vertex $v\in V(C)$ such that $v$ is inside of $C'$.
	Assume, for contradiction, that $C$ is not almost-inside of $C'$.  By
	definition, there must be a vertex $w\in V(C)$ that is outside of $C'$.
	Hence, there is a square $C$ that is not entirely contained in
	$\Gin{C',\pi}$ or $\Gout{C',\pi}$; a contradiction to
	\cref{lem:C4inGinGout}.
\end{proof}

\begin{corollary}\label{cor:uniqueMaxS}
	Let $G$ be a $\pi$-embedded planar median graph.  Then, for all
	$C\in \mathcal{S}(G)$, there is a unique $\preceq_{G,\pi}$-maximal
	element $C'\in \mathcal{S}(G)$.
\end{corollary}

\begin{figure}[t]
	\begin{center}
		\includegraphics[width=.85\textwidth]{algo2.pdf}
	\end{center}
	\caption{A $\pi$-embedded planar median graph $G$ and its forest
		$\mathscr{F}\coloneqq \mathscr{F}(G,\pi)$.  For better
		readability, set $G_i^{\mathrm{out}}(C) \coloneqq \Hout{C,\pi}$ and
		$G_i^{\mathrm{in}}(C) \coloneqq \Hin{C,\pi}$ for $H=G_i$.  The vertices
		of $\mathscr{F}$ correspond to the squares of $G$ and to the vertices
		of $G$ that are not part of a square.  The vertices of $\mathscr{F}$
		are horizontally aligned based on the level in which they occur
		(highlighted by gray lines). According to \cref{alg:TREE-compose}, we
		compute first the subgraph of $G$ that is induced by the vertices (of
		squares) on Level $0$ and we obtain the graph $G_1$.  Since for all
		squares $C$ of $G_1$ the equality $G_1^{\mathrm{out}}(C) = G_1$ holds.
		\cref{lem:decomp} implies that $G_1$ must be a basic QS-graph.  We then
		proceed in the next step to consider the subgraph $H$ of $G$ that is
		induced by the vertices (of squares) on Level $0$ and $1$ that are not
		isolated in $\mathscr{F}$ (cf.\ \Cref{alg:V2}).  Hence, $H$ is induced
		by the vertices $V(G)\setminus (V(C_7) \cup \{b,c,d\})$. The connected
		components of $H$ yield the factors $G_2,G_3$ and $G_4$. Here, $G_2$
		corresponds to the subtree of $\mathscr{F}$ consisting of the edge
		$\{C_1,C_2\}$. According to \cref{alg:TREE-compose}, $C_1$ is the cycle
		in level $0$ which is used to identify $G_1$ and $G_2$, and we obtain
		$G_1\cd{C_1}{C_1}G_2$. For $G_2$ we have $G_2^{\mathrm{in}}(C_1) = G_2$
		as well as $G_2^{\mathrm{out}}(C_2) = G_2$ and thus, by
		\cref{lem:decomp}, $G_2$ is a basic QS-graph.  Proceeding in this
		manner, we obtain the final decomposition
		$G = ((((G_1\cd{C_1}{C_1}G_2 )\cd{C_5}{C_5}G_3) \cd{C_8}{C_8}G_4)
		\cd{C_2}{C_2}G_5) \cd{C_9}{C_9}G_6$ of $G$ into basic QS-graphs.  }
	\label{fig:algo}
\end{figure}

Now, we define the rooted graph
$\mathscr{F}\coloneqq \mathscr{F}(G,\pi)$ for any given
$\pi$-embedded planar median graph $G=(V,E)$. Let
$W\coloneqq V\setminus \Big(\bigcup_{C\in \mc{S}(G)}V(C)\Big)$ be the set
of all vertices of $G$ that are not contained in a square.  The vertex set
of $\mathscr{F}$ is $V(\mathscr{F})=W\cup \mc{S}(G)$ and we add edges in
the following cases:
\begin{itemize}
	\item $\{C,C'\}\in E(\mathscr{F})$ with $C,C'\in \mathcal{S}(G)$ if and
	only if $C\prec_{G,\pi} C'$ and there is no $C''\in \mathcal{S}(G)$ such
	that $C\prec_{G,\pi} C'' \prec_{G,\pi} C'$, and
	\item $\{C,x\}\in E(\mathscr{F})$ with $C\in \mathcal{S}(G)$ and $x\in W$
	if and only if $x$ is inside $C$ and there is no $C'\in \mathcal{S}(G)$
	such that $C' \prec_{G,\pi} C$ and $x$ is inside $C'$.
\end{itemize}

To root this graph $\mathscr{F}$, observe first that if the outer
boundary of $G$ is a square $C$, then there must be a path from $C$ to all
other vertices in $\mathscr{F}$ and thus $C$ is the unique
$\preceq_{G,\pi}$-maximal element for all squares and vertices in
$V(\mathscr{F})$.  In this case, $\mathscr{F}$ must be connected and we
choose $C$ as its root.  If $\mathscr{F}$ is connected but
$\mathcal{S}(G)=\emptyset$, then $\mathscr{F}$ consists of a single vertex
$x\in W$ which is chosen as the root.  If $\mathscr{F}$ is disconnected,
then every connected component $T$ of $\mathscr{F}$ is either a single
vertex $x\in W$ in which case $x$ is chosen as the root of $T$, or $T$
contains vertices in $\mathcal{S}(G)$ in which case the unique
$\preceq_{G,\pi}$-maximal element of the squares that are contained in the
vertex set of $T$ is chosen as the root of $T$.  This unique
$\preceq_{G,\pi}$-maximal element exists due to \cref{cor:uniqueMaxS}.

\begin{lemma}
	Let $G$ be a $\pi$-embedded planar median graph.
	Then, $\mathscr{F}(G,\pi)$ is a rooted forest.  In particular,
	$ \mathscr{F}(G,\pi)$ is a tree if the outer boundary of $G$
	w.r.t.\ $\pi$ is a square.
\end{lemma}
\begin{proof}
	By construction and the arguments preceding this lemma,
	$\mathscr{F}\coloneqq \mathscr{F}(G,\pi)$ is a rooted graph, i.e., all
	its connected components $T$ are rooted at the unique
	$\preceq_{G,\pi}$-maximal element of the squares in $T$ or, in case
	$V(T)=\{x\}\subseteq W$, $T$ is rooted at $x$.  For simplicity, we extend
	in this proof the partial order $\preceq_{G,\pi}$ $\mathcal{S}(G)$ to an
	order $\preceq_{G,\pi}^*$ on $\mathcal{S}(G)\cup W$ by putting
	$C \preceq_{G,\pi}^* C'$ whenever $C \preceq_{G,\pi} C'$ for all
	$C,C\in \mathcal{S}(G)$ and $x \preceq_{G,\pi}^* C$ for all edges
	$\{C,x\}\in E(\mathscr{F})$ with $C\in \mathcal{S}(G)$ and $x\in W$. It
	is easy to verify that $(\mathcal{S}(G)\cup W, \preceq_{G,\pi}^*)$
	remains a partially ordered set.
	
	Now, we assume for contradiction that $\mathscr{F}$ is not a forest.
	Hence, it must contain a cycle. Let $T$ be a connected component 
	that contains a cycle $C_T$. Since $\preceq_{G,\pi}^*$ is a
	partial order on the set $\mathcal S(G)\cup W$ of all squares of $G$
	and all vertices not contained in squares, the cycle $C_T$
	must contain a vertex $v\in V(T)$, such that $v \preceq_{G,\pi}^* w$ for
	all $w\in C_T$. Note, $v$ corresponds either to a square in $G$
	or a vertex that is not contained in a square. 
	Now, consider the two vertices $w$ and $w'$ that are
	adjacent to $v$ in $C_T$. Since $v \preceq_{G,\pi}^* w,w'$, the vertices
	$w$ and $w'$ coincide, by definition of $\preceq_{G,\pi}$, with two
	squares $C$ and $C'$ of $G$, respectively. Furthermore, we have
	$C\neq C'$.
	
	However, $C\prec_{G,\pi}^* C'$ is not possible since, in this case, the
	edge $\{C',v\}$ would not exist in $\mathscr{F}$, by definition.
	Similarly, $C'\prec_{G,\pi}^* C$ is not possible.  However, all vertices of
	one square, say $C$, must be almost-inside $C'$, as otherwise, $C$ is not
	entirely contained in $\Gin{C',\pi}$ and $\Gout{C',\pi}$ and we would
	obtain a contradiction to \cref{lem:C4inGinGout}.  But then,
	$C\prec_{G,\pi}^* C'$; a contradiction.  Hence, $\mathscr{F}$ cannot
	contain cycles and is therefore, a forest.
	
	By the arguments preceding this lemma, if the outer boundary of $G$ is a
	square $C$, then there must be a path from $C$ to all other vertices in
	$\mathscr{F}$ and thus, $\mathscr{F}$ is a tree. By construction, this
	tree is rooted at $C$.
\end{proof}

Since $ \mathscr{F} = \mathscr{F}(G,\preceq_{G,\pi})$ is a forest and based
on the definition of edges $\{C,x\}\in E(T)$ with $C\in \mathcal{S}(G)$ and
$x\in W\coloneqq V\setminus \Big(\bigcup_{C\in \mc{S}(G)}V(C)\Big)$, every
$x\in W$ in $\mathscr{F}$ must be a leaf (i.e., it has degree one in
$\mathscr{F}$) or a singleton (i.e., it is an isolated vertex in
$\mathscr{F}$).

Note that the structure of $\mathscr{F}$ does not completely determine the
structure of $G$, since it only accounts for the ``hierarchy'' of the
nested squares.  As an example, consider
$\mathscr{F}(G,\preceq_{G,\pi})= \mathscr{F}(G',\preceq_{G',\pi'}) =
(\{C_1,C_2\},\emptyset)$ for the graph $G$, resp., $G'$, where all vertices
are located at the outer boundary and where $G$, resp., $G'$, consists
precisely of two squares identified on a single vertex, resp., identified
on a single edge.  The forest $\mathscr{F}$, however, does determine
whether or not $\Gin{C,\pi}\neq G$ for the squares $C\subseteq G$, since
$\Gin{C,\pi}\neq G$ if and only if there are vertices that are contained
outside of $C$ and thus, there must be vertices or squares on the same
level of $C$ in $\mathscr{F}$ or squares above the level of $C$ in the
connected component $T$ of in $\mathscr{F}$ that contains $C$.  In a
similar way, one can determine whether or not $\Gout{C,\pi}\neq G$.

The latter observation can also be applied to subgraphs of $G$ in order to
find basic QS-graphs as follows.  Let us consider a connected induced
subgraph $H\subseteq G$ where $H$ contains only squares and vertices of $G$
that are located in level $i$ and $i+1$ but none of the squares and
vertices of other levels $j$ with $j<i$ and $j>i+1$. Then, for every square
$C$ in $H$ from level $i$ the equality $\Hin{C,\pi} =H$ holds and for every
square $C$ in $H$ from level $i+1$ the equality $\Hout{C,\pi} =H$ holds.  Hence,
for all squares $C\subseteq H$ either $\Hin{C,\pi} =H$ or
$\Hout{C,\pi} =H$.  By \cref{lem:decomp}, $H$ is a basic QS-graph, provided
that $H$ is a median graph (a property that is always satisfied as shown in
the proof of \cref{lem:alg-correct}).  Hence, in order to find a
composition of a planar median graph into basis QS-graphs, we traverse
$\mathscr{F}$ in top-down fashion from level to level and, in principle,
use the connected components of the subgraphs of $G$ that are determined by
the vertices $x$ and squares $C$ on level $i$ and $i+1$ as basic QS-graphs,
see \cref{fig:algo} for an illustrative example. The pseudocode of this
approach is summarized in \Cref{alg:TREE-compose}.

\begin{algorithm}[tbp]
	\caption{\texttt{Ordered Composition of Planar Median Graphs into basic
			QS-graphs.}}
	\label{alg:TREE-compose}
	\begin{algorithmic}[1]
		\Require Graph $G=(V,E)$
		\Ensure  Ordered composition
		$G = (G_1\cd{C^*_i}{C^*_i} G_{2}) .. )\cd{C^*_k}{C^*_k} G_k$
		of basis QS-graphs, if $G$ is a planar median graph and otherwise,
		return ``\texttt{false}''
		\If{$G$ is not planar or a median graph} \label{alg:check1} 
		\Return ``\texttt{false}''
		\ElsIf{$G$ is a square or a tree}\label{alg:check2} 
		\Return $G$ 
		\Else \ Compute planar embedding $\pi$ of $G$
		\label{alg:pi} 
		\EndIf
		\State Compute forest
		$\mathscr{F}\coloneqq \mathscr{F}(G,\pi)$  
		\label{alg:tree}
		\State $j\gets 1$
		\If{$\mathscr{F}$ is disconnected}  \label{alg:if-disc}
		\State  $V'\gets$  subset of vertices of $G$  corresponding to squares and vertices  in $\mathscr{F}$ on level $0$ \label{alg:V1}
		\State $G_j\gets G[V']$	and $j\gets j+1$ 
		\EndIf
		\For{ $i = 1$ to  $\mathit{last\_level}$ of $\mathscr{F}$} \label{alg:for1}
		\State $V'\gets$ subset of vertices of $G$ corresponding to non-isolated squares and vertices in  $\mathscr{F}$ on level $i-1$ and $i$ \label{alg:V2}
		\For{all connected components $H$ in $G[V']$} 
		\State $G_j\gets H$ \label{alg:noS1} 
		\State $C_j\gets$ square from level $i-1$ \label{alg:getS}
		\State $j\gets j+1$ \label{alg:noS2} 
		\EndFor
		\EndFor
		\State \Return
		$G = (G_{1}\cd{C_2}{C_2} G_{2}) \dots )\cd{C_{j-1}}{C_{j-1}} G_{j-1}$	
	\end{algorithmic}
\end{algorithm}

\begin{lemma}\label{lem:alg-correct}
	\Cref{alg:TREE-compose} determines whether a graph $G$ is a planar median
	graph and, in the affirmative case, it returns an ordered decomposition
	of $G$ into basic QS-graphs. Unless $G$ is a square, none of the
	factors is the unit element $C_4$.
\end{lemma}	
\begin{proof}
	Let $G$ be an arbitrary graph. \Cref{alg:check1} ensures that $G$ is
	a planar median graph and we compute in \Cref{alg:pi} a respective
	planar embedding of $G$.  In particular, \Cref{alg:check2} ensures
	that if $G$  is a tree or a square, then $G$ is returned
	and we obtain the trivial composition $G = G$.  Else, $G$ is a cyclic
	planar median graph having at least five vertices. As argued above, the
	forest $\mathscr{F}$ computed in \Cref{alg:tree} is well-defined.
	For better readability, we use in the following the notation $\GinS(*)$
	and $\GoutS(*)$ instead of $\Gin{*,\pi}$ and $\Gout{*,\pi}$, respectively.
	
	Assume that $\mathscr{F}$ is disconnected (\Cref{alg:if-disc}), then the
	outer boundary of $G$ cannot be a square. In particular, $\mathscr{F}$
	must have at least two connected components.  By definition, all
	connected components in $\mathscr{F}$ are trees whose vertices correspond
	to single vertices in $G$ or to squares $C\subseteq G$.  We collect in
	$V'$ all such vertices of $G$ and the vertices of these squares of $G$
	that are in level $0$ of $\mathscr{F}$ (\Cref{alg:V1}) and put
	$G_1 = G[V']$. Let $C_1,\dots,C_k$ be the respective squares that are in
	level $0$ of $\mathscr{F}$. Note, since $G$ is a cyclic median graph, at
	least one such square must exist. 
	Recall that $G$ is planar and $\{K_3,K_{2,3}\}$-free.
	By construction, $G_1 = \cap_{i=1}^k \GoutS(C_i)$, and by \cref{lem:in-out-C4}, $G_1$ is a
	convex subgraph of $G$.
	\cref{lem:basic-median-props}~(4b) implies that $G_1$ is a median graph.
	This, together with the fact that $G^{\mathrm{out}}_1(C)=G_1$ for all squares $C\subseteq G_1$ 
	and \cref{lem:decomp}, implies that  $G_1$ is a
	basic QS-graphs.
	
	Now, we proceed on all levels $i = 1$ to the last level of $\mathscr{F}$
	(\Cref{alg:for1}), which covers also the case that $\mathscr{F}$ is
	connected.  Let us assume we are in some step $i$. In this case, we
	consider the squares and vertices in level $i-1$ and $i$.  We collect in
	$V'$ all such vertices $v$ of $G$ and the vertices of the squares of $G$
	that are in level $i$ and $i-1$ of $\mathscr{F}$ (\Cref{alg:V2}). 
	Now, let $H$ be some connected component of $G[V']$, and
	consider the respective connected of component $T$ in $\mathscr{F}$ that
	corresponds to $H$.  The vertex in level $i-1$ of $T$ must correspond to
	a square $C$ for which all vertices on level $i$ in $T$ are almost-inside
	$C$.  We set $G_j \coloneqq H$. Since $C$ is an outer boundary of $G_j$, we have
	$G^{\mathrm{in}}_j(C) =G_j$.
	All other squares $C'\neq C$ of $G_j$ must
	be inner squares that satisfy $G^{\mathrm{out}}_j(C')=G_j$ by
	construction.
	In case that $G_j$ is a planar median graph, we can apply the
	contraposition of \cref{lem:decomp} to conclude that $G_j$ must be a
	basic QS-graphs.  Clearly, $G_j$ is planar. 
	Thus, it remains to show that
	$G_j$ is a median graph.  
	Recall that $G$ is planar and $\{K_3,K_{2,3}\}$-free.  Let $C'_1,\dots,C'_k$ be the squares
	of $G_j$ that correspond to the vertices of $T$ in level $i$.  By
	construction, $G_j = \GinS(C)\bigcap (\cap_{i=1}^k \GoutS(C'_i))$ and
	\cref{lem:in-out-C4} implies that $G_j$ is a convex subgraph of $G$.
	By \cref{lem:basic-median-props}~(4b), $G_j$ is a median graph.
	By the aforementioned arguments, $G_j$ is a basic QS-graph.
	
	Thus, all graphs $G_j$ computed by \cref{alg:TREE-compose} are basic
	QS-graphs.  We traverse the connected components of $G[V']$ in Step $i$
	according to the subtrees in $\mathscr{F}$ with vertices in level $i-1$
	and $i$. Since we consider in \Cref{alg:V2} only vertices that
	correspond to squares and vertices that are not isolated in $\mathscr{F}$
	on level $i-1$ and $i$, for every connected component
	$G_j\subseteq G[V']$ that is examined in Step $i$ in 
	\Cref{alg:noS1}-\ref{alg:noS2}, the graph $G_j$ is based on the squares
	and vertices that correspond to \emph{adjacent} vertices in $\mathscr{F}$
	that are in level $i-1$ and $i$.  Let $x$ be a vertex in $\mathscr{F}$ on
	level $i$ and assume that $x$ corresponds to the square $C_x\subseteq G$
	in $G_j$.  Moreover, let $V''$ be the set of all such vertices computed
	in Step $i+1$ in \Cref{alg:V2}.  The square $C_x$ will be part of
	some connected component $G_{j'}$ of $G[V'']$. Note, $G_{j'}$ cannot be a
	square, since we consider on level $i$ and $i+1$ non-isolated vertices in
	$\mathscr{F}$.  Otherwise, $G_{j'}$ is not the unit element and all
	remaining squares and vertices distinct from $C_x$ are the adjacent
	vertices of $x$ in $\mathscr{F}$ in level $i+1$.  We set $C_{j'} = C_x$
	(\Cref{alg:getS}).  Since $C_{j'} = C_x$ is contained in the previous
	factor $G_j$ we can ensure that $(\dots) \cd{C_{j'}}{C_{j'}} G_{j'}$ is
	well-defined in each step. The latter arguments also imply that only
	factors that are distinct from the unit-element are used.
\end{proof}

Let $G$ be an outer-planar median graph and $\pi$ be a planar embedding such
that  all vertices of $G$ are incident with the outer face.  Hence,  
$G$ is a tree, a square, or $\mathscr{F}(G,\pi)$ contains only vertices that are located at
level $0$ of $\mathscr{F}(G,\pi)$. In this case, either the
graph $G$ is returned in \Cref{alg:check2} or the \emph{if}-condition
in \Cref{alg:if-disc} is executed and, afterwards, $G$ is immediately
returned. In both cases, \cref{alg:TREE-compose} returns the basic QS-graph
$G$. This, together with the fact that cubes are not outer-planar, implies
\begin{proposition}
	Every outer-planar median graph is a square-graph. 
\end{proposition} 

We proceed with investigating the running time of \Cref{alg:TREE-compose}.

\begin{lemma}
	\Cref{alg:TREE-compose} can be implemented to run in $\mc{O}(|V|\log{|V|})$
	time for every input graph $G=(V,E)$.
\end{lemma}
\begin{proof}
	In the following, let $n=|V|$ and $m=|E|$.  Then $\mc{O}(m+n)$ time
	is required to check whether a graph $G$ is planar \cite{HT:74,CHIBA:85},
	and a median graph \cite{Imrich1999} in \Cref{alg:check1}.  Since a
	planar graph has at most $m\le 3(n-2)$ edges for all $n\geq 3$, we have
	$\mc{O}(m)\subseteq \mc{O}(n)$. Testing for planarity first ensures that no graphs
	that violate this condition are processed further, which implies that
	testing whether $G$ is a median graph can be done in $\mc{O}(n)$ time as well.
	\Cref{alg:check2} can be done in $\mc{O}(n)$ time.  
	
	All squares in a planar graph can be coded efficiently in
	$\mc{O}(m) \subseteq \mc{O}(n)$ time \cite[Thm.\ 20.3 and 20.5]{HammackEtAl2011}.
	Moreover, by \cref{prop:NrSquares}, $G$ has $\mc{O}(n)$ squares.  Hence, all
	squares can be identified in $\mc{O}(n)$ time.
	
	In \Cref{alg:pi}, we compute a planar embedding $\pi$ of $G$ and are, in
	particular, interested in an embedding such that all edges are straight
	lines which can be done in $\mc{O}(n\log n)$ time \cite{FPP:88}.  The
	construction of a forest describing the nesting of the squares and the
	vertices not contained in a square can then be obtained by a modified
	version of the $\mc{O}(n\log n)$ time algorithm solving the polygon
	nesting problem for disjoint (not necessarily convex) polygons
	\cite{Bajaj:90}. In brief, the non-square vertices can be treated as
	``polygons'' consisting of a single vertex and pose no problem. In
	contrast to \cite{Bajaj:90}, we may have squares that share vertices and
	edges. If, during the line-sweep step, a vertex $u$ is encountered that
	is contained in more than one square, one can determine the nesting of
	these incident squares by considering clock-wise ordering of the
	corresponding edges pointing to the right of the sweeping line. If
	multiple squares share the same first edge in this ordering, then their
	nesting is defined by ordering of the second edge. If both edges are
	shared, then the nesting is determined by the horizontal coordinate of
	the fourth point. For $n_x$ squares sharing a point $x$, the nesting can
	thus be computed by sorting the $2n_x$ edges incident with $x$ at most
	thrice, and thus in $\mc{O}(n_x\log n_x)$ time. The total effort for
	disentangling squares with common points thus is also bounded by
	$\mc{O}(n\log n)$.
	
	The forest $\mathscr{F}$ has $\mc{O}(n)$ vertices and edges and can, thus, be
	traversed in $\mc{O}(n)$ time.  Finding induced subgraphs $G[V']$ with $n'$
	vertices and $m'$ edges can be done in $\mc{O}(n'+m')$ time.  In each step $i$
	we have the vertices from graphs level $i-1$ and $i$ with $n_{i-1}$
	vertices and $m_{i-1}$ edges from level $i-1$ and $n_{i}$ vertices and
	$m_{i}$ edges from level $i$. The time needed to compute all the
	subgraphs used in the computation therefore adds up to $\mc{O}(2(m+n))=\mc{O}(n)$.
	
	The overall complexity is therefore $\mc{O}(n\log n)$.
\end{proof}

\begin{figure}[t]
	\begin{center}
		\includegraphics[width=.85\textwidth]{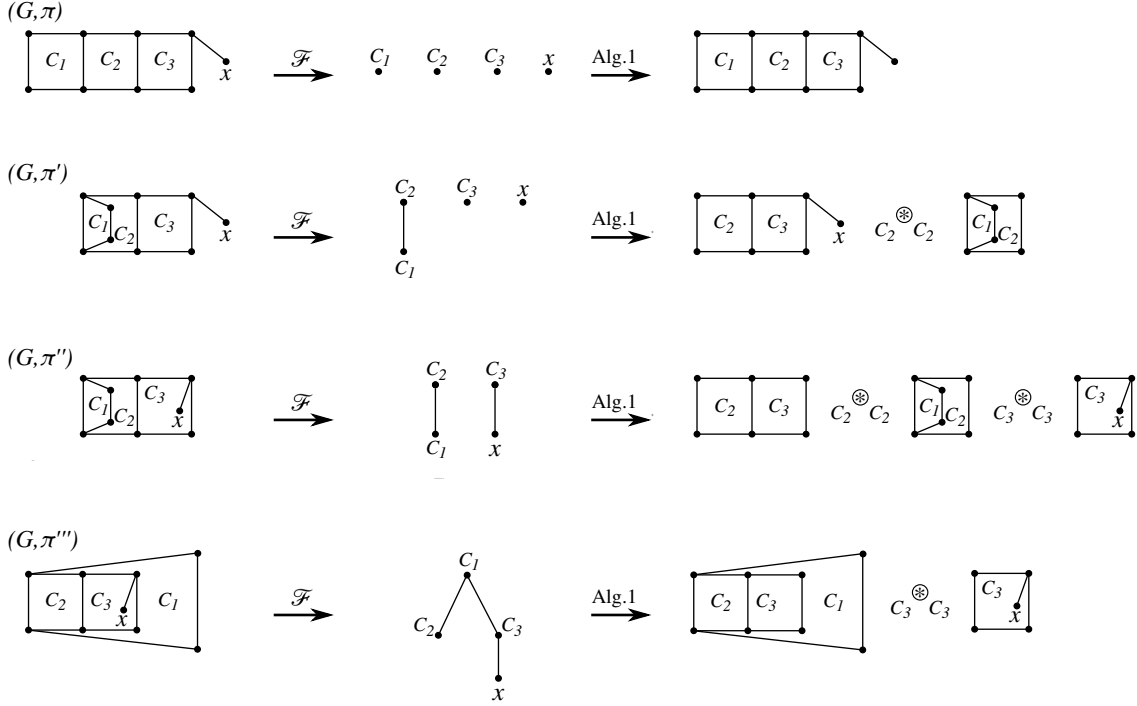}
	\end{center}
	\caption{Four different planar embeddings $\pi$, $\pi'$, $\pi''$, and
		$\pi'''$ of the graph $G$ result in different forests $\mathscr{F}$.
		Depending on the particular embedding, \Cref{alg:TREE-compose} returns
		different solutions: For $(G,\pi)$ we obtain $G=G$; for $(G,\pi')$ we
		obtain $G=G'\cd{C_2}{C_2} H$; for $(G,\pi'')$ we obtain
		$G=(H\cd{C_2}{C_2} H) \cd{C_3}{C_3} G''$ and for
		$(G,\pi''') = G'''\cd{3}{3} G''$.  It can easily be verified that
		the factorization of $(G,\pi'')$ contains irreducible factors only. }
	\label{fig:algo-diff}
\end{figure}

Consider a $\pi$-embedded planar graph
$G$ for which $\GoutS\coloneqq \Gout{C, \pi}$ and
$\GinS\coloneqq \Gin{C, \pi}$ are QS-graphs. In this case, 
$\GoutS$ has an ordered composition of the form
$\GoutS =(\dots (G_1\circledast G_2) \circledast \cdots G_{\ell-1})
\circledast G_{\ell}$ of $\ell\geq 1$ basic QS-graphs. Similarly,
$\GinS = (\dots ( H_1\circledast H_2 ) \circledast \cdots H_{k-1})
\circledast H_{k}$ is composed of $k\geq 1$ basic QS-graphs.  Since
this ordered composition is not associative, we cannot in general write
\begin{equation}\label{eq:order}
\GinS\cd{C}{C} \GoutS = 
(\dots (H_1\circledast H_2) \circledast \cdots H_{k-1}) \circledast H_{k})
\cd{C}{C}\ G_1)\circledast G_2)
\circledast \cdots G_{\ell-1}) \circledast G_{\ell}\,.
\end{equation}
In particular, this expression is not well-defined whenever $C$ is not
contained in $G_1$. \Cref{alg:TREE-compose}, however, makes it
possible to find ordered compositions for both $\GoutS$ and $\GinS$ such
that we can write $G$ in the form of Equ.\ \eqref{eq:order}.

\begin{proposition}\label{prop:find-factor-ordering}
	Let $G$ be a connected $\pi$-embedded planar graph and $C\subseteq G$ be
	a square. Suppose that $\GoutS\coloneqq \Gout{C,\pi}$ and $\GinS\coloneqq \Gin{C,\pi}$ are QS-graphs, and
	the factorization
	$\GoutS =(\dots (G_1\circledast G_2) \circledast \cdots G_{\ell-1})
	\circledast G_{\ell}$ of $\GoutS$ and
	$\GinS = (\dots ( H_1\circledast H_2 ) \circledast \cdots H_{k-1})
	\circledast H_{k}$ of $\GinS$ has been computed with
	\Cref{alg:TREE-compose} w.r.t.\ the $(G,\pi)$-induced planar embedding of
	$\GoutS$ and $\GinS$, respectively.  Then,
	$(\dots (H_1\circledast H_2) \circledast \cdots H_{k-1}) \circledast
	H_{k})\circledast\ G_1 )\circledast G_2) \circledast \cdots G_{\ell-1})
	\circledast G_{\ell}$ is well-defined and yields a factorization of $G$
	into basic QS-graphs.
\end{proposition}
\begin{proof}
	\Cref{thm:if-inoutside->QSthree<->cubesq} implies that $G$ is a QS-graph.
	The outer boundary of $\GinS$ w.r.t.\ $(G,\pi)$-induced embedding is the
	square $C$ and thus, $\mathscr{F}(\GinS,\pi)$ is
	connected and rooted at $C$. The square $C$ serves an inner boundary in
	$\GoutS$ w.r.t.\ $(G,\pi)$-induced embedding and there are no further
	vertices of $\GoutS$ inside $C$. Hence $C$ is a leaf in
	$\mathscr{F}(\GoutS,\pi)$. Now, it is easy to see that
	$\mathscr{F}(G,\pi)$ is identical to the forest that is
	obtained from $\mathscr{F}(\GoutS,\pi)$ and
	$\mathscr{F}(\GinS,\pi)$ by identifying the leaf $C$ of
	$\mathscr{F}(\GoutS,\pi)$ with the root $C$ of
	$\mathscr{F}(\GinS,\pi)$.  In a similar fashion as in
	\Cref{alg:TREE-compose}, we traverse $\mathscr{F}(G,\pi)$ but
	use first only the vertices that are contained in
	$\mathscr{F}(\GoutS,\pi)$ in the same order that yields
	the factorization
	$\GoutS =(\dots (G_1\circledast G_2) \circledast \cdots G_{\ell-1})
	\circledast G_{\ell}$ provided by \Cref{alg:TREE-compose} applied on
	$\GoutS$ and afterfwards, we traverse the subtree
	$\mathscr{F}(\GinS,\pi)$ in the same order as in
	\Cref{alg:TREE-compose} to obtain the factorization
	$\GinS = (\dots ( H_1\circledast H_2 ) \circledast \cdots H_{k-1})
	\circledast H_{k}$.  This yields the factorization
	$G = (\dots (H_1\circledast H_2) \circledast \cdots H_{k-1}) \circledast
	H_{k})\cd{C}{C}\ G_1 )\circledast G_2) \circledast \cdots G_{\ell-1})
	\circledast G_{\ell}$.
\end{proof}

The result of \Cref{alg:TREE-compose} depends crucially on the chosen
planar embedding. To see this, consider the vertex $b$ in the graph $G$
shown in \Cref{fig:algo}. Placing $b$ into the face bounded by the square
$C_7$ yields an additional factor that is isomorphic to a square to which
an additional vertex is attached.  Another example is shown in
\cref{fig:algo-diff}.

By \cref{thm:if-inoutside->QSthree<->cubesq} $G=G_1\circledast G_2$
is a QS-graph whenever $G_1$ and $G_2$ are QS-graphs.
We say that a QS-graph $G$ is \emph{irreducible} if $G=G_1\circledast G_2$
implies that $G_1$ or $G_2$ is the unit element, i.e., a square.
Irreducible QS-graphs are for example the ``domino'' $P_3\Box K_2$ or a
square to which a single edge is attached.  The composition of $(G,\pi'')$
in \cref{fig:algo-diff} consists of irreducible QS-graphs only.  The
observations above imply that basic QS-graphs are neither necessarily
irreducible nor that there is a unique way to decompose a planar median
into basic QS-graphs.  We suspect, however, that the following statement is
true:
\begin{conjecture}
	Every planar median graph $G$ has a unique composition
	$G=G_1\circledast \cdots \circledast G_k$, $k\geq 1$ of irreducible
	QS-graphs $G_i$, $1\leq i\leq k$ up to isomorphism and possible re-order
	of the factors.
\end{conjecture}

\section{Summary and Outlook}
\label{sec:sum}

In this contribution, we have provided novel characterizations for planar median graph.
\Cref{thm:planar-mg<->iso-cycles} makes use of forbidden subgraphs and the structure
of their isometric cycles. \Cref{thm:if-inoutside->mg<->pmg}, furthermore, shows that
it is sufficient to consider the induced subgraphs $\Gin{C,\pi}$ and $\Gout{C,\pi}$
almost-inside and almost-outside of an arbitrary square $C\subseteq G$. A more
constructive characterization is obtained in terms of the gluing operation
$\circledast$ for QS-graphs that stepwise identifies square-boundaries along which
two graphs are identified. \Cref{thm:mpg<->Qthree} shows that planar median graphs
are exactly the union of QS-graphs and trees. The operation $\circledast$ corresponds
to a specific amalgamation of graphs and provides a corresponding characterization of
planar median graphs by amalgamation (cf.~\cref{thm:amalgam}). The structure of
QS-graphs leads to an $\mc{O}(n\log(n))$ time algorithm that computes an ordered
composition $G=G_1\circledast \cdots \circledast G_k$ into square-graphs and cubes
for a planar median graph $G$ with $n$ vertices.

It would be interesting to know if these results might be extended to planar partial
cubes (noting that any median graph is a partial cube). As with median graphs,
partial cubes arise by isometric expansions \cite{Chepoi:88}, and planar partial
cubes can be characterized by an expansion procedure \cite{Desgranges:17} (see also
\cite{Peterin:08}). The $\circledast$ operation is closely related to the
non-crossing 2-face expansion employed in \cite{Desgranges:17} used to characterize
partial cubes: $H$ is a 2-face expansion of $G$ if $G_1$ and $G_2$ have plane
embeddings such that $G' \coloneqq G_1 \cap G_2$ lies on a face in both the
respective embeddings. If $G'$ is an edge $e=uv$, then $u$ and $v$ are trivially
located on the same face in $G_1$ and $G_2$. Restricting $G'$ to be an edge leads to
a definition of a ``restricted 2-face expansion'' that can be expressed in terms of
our gluing operation: We expand only $G_1$ on this edge (to get a square $C$) and
only expand $G_2$ on this edge (to get a square $C'$) and then set $H = G_1
\cd{C}{C'} G_2$. Every restricted 2-face expansion can therefore be expressed as
``expand single edges in $G_1$ and $G_2$ and glue together $G_1$ and $G_2$ along the
resulting squares via the $\circledast$ operation''. It could therefore be worth
while investigating if a variant of the $\circledast$ operation could be used to give
new insights into the structure of planar partial cubes.

In another direction, define a planar median graph $G$ as \emph{irreducible} if $G =
G_1\circledast G_2$ implies that $G_1$ or $G_2$ is the unit element, i.e., a square.
It would be interesting to understand whether one can decompose a given planar median
graph into irreducible factors in polynomial-time. Moreover, does every planar median
graph admit a unique composition into irreducible factors? \Cref{alg:TREE-compose}
depends crucially on the particular planar embedding of $G$. Does
\cref{alg:TREE-compose} yield the same factors (possibly in a different order) if the
planar embeddings differ only by the choice of the outer boundary? Answers to the
latter question would possibly provide an avenue to determine the irreducible factors
-- at least for 3-connected planar median graphs, since all their planar embedding
are equivalent. Moreover, one may ask how different forests
$\mathscr{F}(G,\preceq_{G,\pi})$ and $\mathscr{F}(G,\preceq_{G,\pi'})$ are related to
each other for different embeddings $\pi$ and $\pi'$ in the case that $G$ has unique
composition of irreducible QS-graphs?

Finally, \Cref{thm:planar-mg<->iso-cycles} provides a characterization of planar
median graphs in terms of forbidden subgraphs and the structure of their isometric
cycles. Note, there is no forbidden subgraph characterization of planar median
graphs, since the property of being a median graph is not hereditary. However, it is
natural to ask whether a planar median graph can be solely characterized amongst
median graphs in terms of a collection of forbidden subgraphs or minors. A good
starting point for answering this question could be to understand how either 2-face
expansions, convex face expansions or square-boundary amalgamations might shed new
light on the forbidden subgraph theorem for square-graphs mentioned above in
Proposition~\ref{prop:square-from-median}, with the view to extending these
considerations to planar median networks.

\bibliographystyle{plainnat}
\bibliography{planar-median-graphs}

\end{document}